\documentclass{amsart}

\usepackage[colorlinks=true, urlcolor=black, citecolor=black, linkcolor=black, hyperfootnotes=true]{hyperref}
\usepackage{amssymb}
\usepackage{mathtools}
\usepackage{aliascnt}%for theorem counters
\usepackage{graphicx}

\newcommand*{\U}{\smallsmile}
\newcommand*{\pp}{\prec\!\!\prec}
\newcommand*{\qq}{\succ\!\!\succ}

\numberwithin{equation}{section}

\newtheorem*{thm*}{Theorem}

\newtheorem{thm}{Theorem}[section]

\newaliascnt{prp}{thm}
\newtheorem{prp}[prp]{Proposition}
\aliascntresetthe{prp}

\newaliascnt{cor}{thm}
\newtheorem{cor}[cor]{Corollary}
\aliascntresetthe{cor}

\newaliascnt{lem}{thm}
\newtheorem{lem}[lem]{Lemma}
\aliascntresetthe{lem}

\theoremstyle{definition}

\newaliascnt{dfn}{thm}
\newtheorem{dfn}[dfn]{Definition}
\aliascntresetthe{dfn}

\newaliascnt{qst}{thm}

\aliascntresetthe{qst}

\newaliascnt{xpl}{thm}
\newtheorem{xpl}[xpl]{Example}
\aliascntresetthe{xpl}

\newaliascnt{rmk}{thm}
\newtheorem{rmk}[rmk]{Remark}
\aliascntresetthe{rmk}

\author{Tristan Bice}
\email{tristan.bice@gmail.com}
\thanks{The first author is supported by IMPAN (Poland)}

\author{Charles Starling}
\email{cstar@math.carleton.ca}
\thanks{The second author is supported by a Carleton University internal research grant}
\thanks{This collaboration began at the Fields Institute ``Workshop on Dynamical Systems and Operator Algebras'' and ``Workshop on New Directions in Inverse Semigroups'' held at the University of Ottawa in May-June 2016.}


\title[Non-Commutative Locally Compact Locally Hausdorff Stone Duality]{General Non-Commutative Locally Compact Locally Hausdorff Stone Duality}

\begin{document}

\begin{abstract}
We extend the classical Stone duality between zero dimensional compact Hausdorff spaces and Boolean algebras.  Specifically, we simultaneously remove the zero dimensionality restriction and extend to \'{e}tale groupoids, obtaining a duality with an elementary class of inverse semigroups.
\end{abstract}

\maketitle

\section*{Introduction}

\subsection*{Motivation}

There has been a recent surge of interest in extending classical Stone dualities between certain lattices and topological spaces to corresponding non-commutative objects like inverse semigroups and \'{e}tale groupoids. This line of investigation started with work of Kellendonk \cite{Kellendonk1997} and Lenz \cite{Lenz2008} who aimed to reconstruct tiling groupoids (which are zero-dimensional) from a certain inverse semigroup of its compact bisections defined from the geometry of the tiling in question. Exel \cite{Exel2010} then showed one could always reconstruct a zero-dimensional \'etale groupoid from its inverse semigroup of compact bisections. A number of recent papers, notably by Kudryavtseva, Lawson, Lenz and Resende, have generalized these to various settings.  For example \cite{Lawson2012} and \cite{KudryavtsevaLawson2016} extend the classical Stone duality between generalized Boolean algebras and zero-dimensional locally compact Hausdorff topological spaces, while \cite{Resende2007}, \cite{LawsonLenz2013} and \cite{KudryavtsevaLawson2017} extend the duality between spatial frames/locales and sober topological spaces.  However, even in the classical commutative cases, both these dualities have their drawbacks.  Specifically, zero-dimensional spaces are often too restrictive for applications, for example when we want use \'{e}tale groupoids to define C*-algebras with few projections.  On the other hand, while frames describe very general topological spaces, they are often too big, usually uncountable even when the spaces they describe are countable (e.g. the frame of open sets of $\mathbb{Q}$).  In model theoretic terms the problem is that, while Boolean algebras are first order structures, frames are second order structures, relying as they do on infinitary operations, namely infinite joins.

Our goal is to show that it is possible to get the best of both worlds, removing the zero-dimensionality restriction while still obtaining a duality with a first order structure.  More precisely we show that certain bases of general locally compact locally Hausdorff \'{e}tale groupoids are dual to a natural first order finitely axiomatizable class of inverse semigroups.  The key is to consider not just the canonical order $\leq$ in the semigroup but also the `rather below' relation $\prec$, which allows one to recover compact containment on the corresponding basis elements.

In fact a similar idea, at least in the commutative case, already appears in \cite{Shirota1952}, which has led to the study of `compingent algebras', i.e. Boolean algebras together with an extra proximal neighbourhood relation $\ll$ - see \cite{deVries1962} and \cite{BaayendeRijk1963}.  These are dual to algebras of regular open sets in compact Hausdorff spaces, where the join operation is given by $\overline{O\cup N}^\circ$.  In contrast, we consider general open sets where the join operation is given simply by the union $O\cup N$.  We feel this is the more natural operation to consider, although it would also be interesting to see if a similar non-commutative extension of Shirota/de Vries duality could be obtained with appropriately defined `compingent semigroups'.

\subsection*{Outline}

In \cite[\S1-\S4]{BiceStarling2016} we obtained a duality between basic lattices and certain bases of locally compact Hausdorff spaces.  The first order of business is to extend this from Hausdorff to locally Hausdorff spaces, which is the content of \autoref{secAuxiliaryRelations}-\autoref{secFunctoriality}.  Apart from the inherent interest in generalization, there are natural examples of \'{e}tale groupoids which are only locally Hausdorff and we would like our duality to cover such groupoids.  As a union of Hausdorff subsets is not necessarily Hausdorff, this means we can no longer assume our bases are closed under arbitrary finite unions.  We also no longer assume they are closed under finite intersections.  Again the motivation for this comes from \'{e}tale groupoids where often one deals with compact (open) bisections, which are not closed under taking intersections in the non-Hausdorff case.  On a more commutative level, this also allows us to answer \cite[Question 3.6]{BiceStarling2016} in the affirmative.

On the order theoretic side of things, this means we are no longer dealing with lattices but only conditional $\vee$-semilattices (i.e. only bounded pairs have joins \textendash\, see \cite[Definition I-4.5]{GierzHofmannKeimelLawsonMisloveScott2003}).  To extend the theory to such posets, in \autoref{subsecDistributivity} we investigate an appropriate version of distributivity for a general auxiliary relation $\prec$ (see \cite[Definition I-1-11]{GierzHofmannKeimelLawsonMisloveScott2003}), as well as a weaker notion in \autoref{subsecDecomposition}.  In \autoref{subsecHausdorff} we define and investigate an order theoretic analog $\U$ of the `Hausdorff union' relation.  From \autoref{subsecRatherBelow} onwards, we restrict our attention to the rather below relation $\prec$, which is the order theoretic analog of `compact containment'.

The abstract counterparts to the bases we consider are the basic posets we define in \autoref{secBasicPosets}.  We first investigate a couple of properties specific to basic posets in \autoref{subsecBasicProperties} and then provide some Boolean examples in \autoref{subsecBooleanExamples}.  We finish this section with \autoref{Ubases}, showing that $\cup$-bases of locally compact locally Hausdorff spaces are indeed $\U$-basic posets.  It might even be helpful to look at \autoref{Ubases} first to get some idea of the significance of the posets and relations we are considering. 

To obtain the other direction of the duality, we need to examine filters.  The first step is to show that characterizations of ultrafilters for Boolean algebras extend to basic posets, as shown in \autoref{ultrachars}.  With the help of the key \autoref{abC}, we then show in \autoref{SpacesFromPosets} how basic posets indeed become bases of their $\prec$-ultrafilter spaces, which are always locally compact locally Hausdorff spaces.

This completes duality between $\cup$-bases and $\U$-basic posets, at least as far objects in the respective categories are concerned.  In \autoref{secFunctoriality}, we show that this duality is also functorial for (even partially defined) continuous maps between spaces and certain relational basic morphisms between basic posets.  The extension to `non-commutative posets', i.e. inverse semigroups, and `non-commutative spaces', i.e. \'{e}tale groupoids, is then dealt with in \autoref{InverseSemigroups} and \autoref{Groupoids}.

\subsection*{Summary}\label{Summary}

The classic Stone duality is usually stated as `zero-dimensional locally compact Hausdorff spaces are dual to generalized Boolean algebras'.  But the duality really concerns bases rather than the spaces themselves.  To make this explicit, let us call a basis of a topological space a \emph{clopen $\cup$-basis} if it consists of compact\footnote{Note for us `compact' just means `every open cover has a finite subcover' or, equivalently, `every net has a convergent subnet', i.e. we do not require compact sets to be Hausdorff.} clopen subsets and is closed under finite unions $\cup$.  In fact, it is not hard to see that a clopen $\cup$-basis must actually consist of \emph{all} compact clopen subsets.  The following is then a slightly more accurate summary of classic Stone duality.

\begin{thm*}[Stone]
Clopen $\cup$-bases are dual to generalized Boolean algebras.
\end{thm*}

More precisely, every clopen $\cup$-basis is a generalized Boolean algebra, when considered as a poset with respect to inclusion $\subseteq$, and conversely every generalized Boolean algebra can be represented as a clopen $\cup$-basis of some topological space.  Of course, to actually have a clopen ($\cup$-)basis the topological space must be zero-dimensional locally compact Hausdorff, in which case the clopen $\cup$-basis is unique, which is why such spaces and bases are usually conflated.

We consider more general $\cup$-bases \textendash\, see \autoref{Ubasisdef}.  Basically, a $\cup$-basis consists of relatively compact Hausdorff subsets and must be closed under finite unions whenever possible.  Thus the spaces that have $\cup$-bases are precisely the locally compact locally Hausdorff spaces \textendash\, see \autoref{LCLHchars} \textendash\, the key point being that they no longer need to be zero-dimensional.  On the abstract side of things, we consider `$\U$-basic posets' \textendash\, see \autoref{BasicUPoset} \textendash\, which generalize Boolean algebras \textendash\, see \autoref{subsecBooleanExamples}.  Our duality can be summarized as follows.

\begin{thm*}
$\cup$-bases are dual to $\U$-basic posets.
\end{thm*}

\begin{proof}
By \autoref{Ubases}, every $\cup$-basis is a $\U$-basic poset with respect to the inclusion ordering.  Conversely, every $\U$-basic poset can be represented as $\cup$-basis of its $\prec$-ultrafilter space, by \autoref{SpacesFromPosets}.  Moreover, if this $\prec$-ultrafilter space representation is applied to a $\cup$-basis, we recover the original space, by \autoref{SubsetUltrafilters}.
\end{proof}

We also have the following non-commutative extension to certain bases of \'{e}tale groupoids and certain kinds of inverse semigroups.

\begin{thm*}
$\cup$-\'{e}tale bases are dual to $\simeq$-basic semigroups.
\end{thm*}

\begin{proof}
By \autoref{EtaleBases}, every $\cup$-\'{e}tale basis is a $\simeq$-basic semigroup.  Conversely, every $\simeq$-basic semigroup can be represented as a $\cup$-\'{e}tale basis of its $\prec$-ultrafilter groupoid, by \autoref{SemigroupRep}.  Moreover, if this $\prec$-ultrafilter groupoid representation is applied to a $\cup$-basis, we recover the original groupoid, by \autoref{UltrafilterMultiplication}.
\end{proof}

The classic Stone duality is also functorial with respect to the appropriate morphisms, which can be summarized as follows.

\begin{thm*}[Stone]
Continuous maps are dual to Boolean homomorphisms.
\end{thm*}

More precisely, any continuous map $F:G\rightarrow H$, for zero dimensional compact Hausdorff spaces $G$ and $H$, yields a Boolean homomorphism $\pi:\mathrm{Clopen}(H)\rightarrow\mathrm{Clopen}(G)$ defined by
\[\pi(O)=F^{-1}[O],\]
and conversely every Boolean homomorphism $\pi:\mathrm{Clopen}(H)\rightarrow\mathrm{Clopen}(G)$ arises in this way from some continuous map $F$.

Here we can not simply replace $\mathrm{Clopen}(G)$ and $\mathrm{Clopen}(H)$ with $\cup$-bases $S$ and $T$ as these are not uniquely defined by $G$ and $H$, so there is no guarantee that $F^{-1}$ will take elements of $T$ to elements of $S$.  Instead of homomorphisms, we consider relational `basic morphisms' $\sqsubset$ \textendash\, see \autoref{BasicMorph} \textendash\, arising from partial continuous maps $F$ as follows.
\[O\sqsubset N\qquad\Leftrightarrow\qquad O\subseteq F^{-1}[N].\]

\begin{thm*}
Partial continuous maps are dual to basic morphisms.
\end{thm*}

\begin{proof}
Every partial continuous map defines a basic ($\vee$-)morphism, by \autoref{ctsphi}.  Conversely, every basic morphism defines a partial continuous map on the corresponding $\prec$-ultrafilter spaces, by \autoref{Stonects}.
\end{proof}

\section{Auxiliary Relations}\label{secAuxiliaryRelations}
We will deal with quite general posets that are not lattices or even semilattices.  However they will still satisfy a version of distributivity with respect to a relation $\prec$ which is auxiliary to $\leq$ in the sense of \cite[Definition I-1-11]{GierzHofmannKeimelLawsonMisloveScott2003}, namely the rather below relation introduced in \autoref{subsecRatherBelow}.

But until then it will be more convenient to consider general auxiliary relations.

\begin{center}
\textbf{From now on assume $S$ is a poset and $\prec\ \subseteq\ \leq$ is an auxiliary relation, i.e.}
\end{center}
\[\label{Auxiliarity}\tag{Auxiliarity}a\leq b\prec c\leq d\qquad\Rightarrow\qquad a\prec d.\]
In particular, if $a\prec b\prec c$ then $a\prec b\leq c$ so $a\prec c$, i.e. $\prec$ is transitive.  However, $\prec$ need not be reflexive.  Indeed, the only reflexive auxiliary relation is $\leq$ itself.

Incidentally, while posets are usually denoted by letters like $\mathbb{P}$, we use $S$ to emphasize that we are primarily interested in posets coming from inverse semigroups \textendash\, see \autoref{InverseSemigroups}.

\subsection{Distributivity}\label{subsecDistributivity}

\begin{dfn}
We say $S$ is \emph{$\prec$-distributive} if, whenever $b,c\in S$ have a join $b\vee c$,
\[\label{precDistributivity}\tag{$\prec$-Distributivity}a\leq b\vee c\quad\Leftrightarrow\quad\forall a'\prec a\ \exists b'\prec b\ \exists c'\prec c\ (a'\prec\,b'\vee c'\,\prec a).\]
\end{dfn}

At first sight this may seem like a strange notion of distributivity.  Indeed, we do not know if any analog of distributivity for non-reflexive transitive relations has been considered before.  However, the reflexive case does reduce to something more familiar.  Specifically, for $\leq$-distributivity we can take $a'=a$ to obtain
\[\label{leqDistributivity}\tag{$\leq$-Distributivity}a\leq b\vee c\qquad\Leftrightarrow\qquad\exists b'\leq b\ \exists c'\leq c\ (a=b'\vee c'),\]
which is the usual notion of distributivity for $\vee$-semilattices.  Also, \eqref{precDistributivity} has several consequences familiar from domain theory.  For example, the $\Rightarrow$ part of \eqref{precDistributivity} in the $a=b=c$ case yields
\[\label{Interpolation}\tag{Interpolation}a\prec b\qquad\Rightarrow\qquad\exists c\ (a\prec c\prec b).\]
Then the $\Leftarrow$ part of \eqref{precDistributivity} can be expressed in a simpler form.

\begin{prp}\label{LeftprecDistributivity}
If \eqref{Interpolation} holds then the $\Leftarrow$ part of \eqref{precDistributivity} is equivalent to either of the following.
\begin{align}
\label{Approximation}\tag{Approximation}\forall c\prec a\ (c\leq b)\qquad&\Rightarrow\qquad a\leq b.\\
\label{LowerOrder}\tag{Lower Order}\forall c\prec a\ (c\prec b)\qquad&\Rightarrow\qquad a\leq b.
\end{align}
\end{prp}

\begin{proof} Even without \eqref{Interpolation}, we have
\begin{equation}\label{Approx=>}
\eqref{Approximation}\quad\Rightarrow\quad\eqref{LowerOrder}\quad\Rightarrow\quad(\Leftarrow\text{ in \ref{precDistributivity}}).
\end{equation}
Indeed, the first $\Rightarrow$ is immediate from $\prec\ \subseteq\ \leq$.  For the second $\Rightarrow$, say \eqref{LowerOrder} and right side of \eqref{precDistributivity} hold.  So, for all $a'\prec a$, we have $b'\prec b$ and $c'\prec c$ with $a'\prec b'\vee c'\leq b\vee c$ and hence $a'\prec b\vee c$, by \eqref{Auxiliarity}.  As $a'$ was arbitrary, \eqref{LowerOrder} yields $a\leq b\vee c$, as required.

Lastly, assume \eqref{Interpolation} and the $\Leftarrow$ part of \eqref{precDistributivity} hold.  To see that \eqref{Approximation} then holds, take $a$ and $b$ satisfying $\forall c\prec a\ (c\leq b)$.  If $a'\prec a$ then \eqref{Interpolation} yields $b',b''\in S$ with $a'\prec b'\prec b''\prec a$.  By assumption, $b''\leq b$ so $a'\prec b'\prec b$, by \eqref{Auxiliarity}.  This shows that the right side of \eqref{precDistributivity} is satisfied for $b=c$ so the $\Leftarrow$ part yields $a\leq b$.
\end{proof}

The $\Rightarrow$ part of \eqref{precDistributivity} for $a=b\vee c$ also yields
\[\label{Shrinking}\tag{Shrinking}a\prec b\vee c\qquad\Rightarrow\qquad\exists b'\prec b\ \exists c'\prec c\ (a\leq b'\vee c').\]
Conversely, using this we can get \eqref{precDistributivity} from \eqref{leqDistributivity}.

\begin{prp}
If $\prec$ satisfies \eqref{Interpolation}, \eqref{LowerOrder} and \eqref{Shrinking},
\[\eqref{leqDistributivity}\qquad\Rightarrow\qquad\eqref{precDistributivity}.\]
\end{prp}

\begin{proof}
Assume \eqref{Interpolation}, \eqref{LowerOrder}, \eqref{Shrinking} and \eqref{leqDistributivity}.  The $\Leftarrow$ part of \eqref{precDistributivity} is then immediate from \autoref{LeftprecDistributivity}.  For the $\Rightarrow$ part, say we are given $a'\prec a\leq b\vee c$.  By \eqref{Interpolation}, we have $a''\in S$ with $a'\prec a''\prec a$.  By \eqref{Auxiliarity}, $a''\prec b\vee c$.  By \eqref{Shrinking}, we have $b'\prec b$ and $c'\prec c$ with $a''\leq b'\vee c'$.  By \eqref{leqDistributivity}, we have $b''\leq b'$ and $c''\leq c'$ with $a''=b''\vee c''$.  Thus \eqref{Auxiliarity} yields
\[b''\prec b,\quad c''\prec c\quad\text{and}\quad a'\prec b''\vee c''\prec a,\]
i.e. the right side of \eqref{precDistributivity} holds, as required.
\end{proof}

The following result generalizes the fact that, in distributive lattices, not only do meets distribute over joins but also joins distribute over meets.

\begin{prp}\label{precabcd}
If $S$ is $\prec$-distributive and both $b\vee c$ and $b\vee d$ exist then
\[a\leq b\vee c\quad\text{and}\quad a\leq b\vee d\qquad\Rightarrow\qquad\forall a'\prec a\ \exists e\prec c,d\ (a'\prec b\vee e).\]
\end{prp}

\begin{proof}
Given $a'\prec a$, \eqref{precDistributivity} applied to $a\leq b\vee c$ yields $b'\prec b$ and $c'\prec c$ with $a'\prec b'\vee c'\prec a$.  By \eqref{precDistributivity} applied to $a'\prec b'\vee c'\leq b'\vee c'$, we have $c''\prec c'$ with $a'\prec b'\vee c''$.  As $c''\prec c'\leq a\leq b\vee d$, applying \eqref{precDistributivity} again yields $b''\prec b$ and $e\prec d$ with $c''\prec b''\vee e\prec c'$.   By \eqref{Auxiliarity}, $e\prec c$ and $a'\prec b\vee e$, as $a'\prec b'\vee c''\leq b'\vee b''\vee e\leq b\vee e$.
\end{proof}

%As $c'\leq b'\vee c'\prec a\leq b\vee d$, \eqref{Auxiliarity} yields $c'\prec b\vee d$.  Applying \eqref{precDistributivity} again yields $b''\prec b$ and $d'\prec d$ with $c'\prec b''\vee d'$.

\subsection{Decomposition}\label{subsecDecomposition}

\begin{dfn}\label{Decomposition}
We say $S$ has \emph{$\prec$-decomposition} if, whenever $b,c\in S$ have a join,
\[\label{precDecomposition}\tag{$\prec$-Decomposition}a\leq b\vee c\qquad\Leftrightarrow\qquad a=\bigvee\{a'\prec a:a'\prec b\text{ or }a'\prec c\}.\]
\end{dfn}

In particular,
\[\label{leqDecomposition}\tag{$\leq$-Decomposition}a\leq b\vee c\qquad\Leftrightarrow\qquad a=\bigvee\{a'\leq a:a'\leq b\text{ or }a'\leq c\}\]
can be restated as follows whenever $a$ has a meet with both $b$ and $c$:
\[\label{leqDecomposition'}\tag{$\leq$-Decomposition$'$}a\leq b\vee c\qquad\Leftrightarrow\qquad a=(a\wedge b)\vee(a\wedge c).\]
Again, this is the familiar notion of distributivity for lattices.
%\label{leqDecomposition}\tag{$\leq$-Decomposition}

%If $a$ has a meet with both $b$ and $c$ then $\leq$-decomposition becomes
%\[\label{leqDecomposition}\tag{$\leq$-Decomposition}a\leq b\vee c\qquad\Leftrightarrow\qquad a=(a\wedge b)\vee(a\wedge c),\]
%which is again the familiar notion of distributivity for lattices.

Note that it suffices to verify $\Rightarrow$ in \eqref{precDecomposition}, as the $\Leftarrow$ part is immediate from $\prec\ \subseteq\ \leq$.  Likewise, $\prec\ \subseteq\ \leq$ immediately yields $\geq$ on the right side so we can rewrite \eqref{precDecomposition} equivalently as follows.
\[\tag{$\prec$-Decomposition}a\leq b\vee c\qquad\Rightarrow\qquad a\leq\bigvee\{a'\prec a:a'\prec b\text{ or }a'\prec c\}.\]
Also note that \eqref{Approximation} is just \eqref{precDecomposition} with $a=b=c$, i.e.
\[\tag{Approximation}\forall a\in S\ (a=\bigvee_{b\prec a}b).\]
In fact, \eqref{precDecomposition} itself follows from \eqref{precDistributivity}.

\begin{prp}\label{Distributivity=>Decomposition}
We have the following general implications.
\[\eqref{precDistributivity}\qquad\Rightarrow\qquad\eqref{precDecomposition}\qquad\Rightarrow\qquad\eqref{leqDecomposition}.\]
\end{prp}

\begin{proof}
For the first $\Rightarrow$, say $a\leq b\vee c$.  For every $a'\prec a$, \eqref{precDistributivity} yields $b'\prec b$ and $c'\prec c$ with $a'\prec b'\vee c'\prec a$.  By \eqref{Auxiliarity}, $b',c'\prec a$.  Also $a'\leq b'\vee c'$, as $\prec\ \subseteq\ \leq$.  By \autoref{LeftprecDistributivity}, \eqref{precDistributivity} yields \eqref{Approximation} so
\[a=\bigvee_{a'\prec a}a'\leq\bigvee\{b'\vee c'\prec a:b'\prec b\text{ and }c'\prec c\}\leq\bigvee\{a''\prec a:a''\prec b\text{ or }a''\prec c\},\]
i.e. \eqref{precDecomposition} holds.  The second $\Rightarrow$ is immediate from $\prec\ \subseteq\ \leq$.
\end{proof}

\begin{cor}\label{precDistributivity=>leqDistributivity}
If $S$ is a $\wedge$-semilattice then
\[\eqref{precDistributivity}\qquad\Rightarrow\qquad\eqref{leqDistributivity}\qquad\Leftrightarrow\qquad\eqref{leqDecomposition}.\]
\end{cor}

\begin{proof}
The $\Rightarrow$ follows from \autoref{Distributivity=>Decomposition}.  The $\Leftrightarrow$ follows from the equivalent of \eqref{leqDecomposition} given above when $a$ has a meet with $b$ and $c$.
\end{proof}

The following is an analog of \autoref{precabcd}.

\begin{prp}\label{veeoverwedge}
If $S$ has $\prec$-decomposition and both $b\vee c$ and $b\vee d$ exist then
\[a\leq b\vee c\quad\text{and}\quad a\leq b\vee d\qquad\Rightarrow\qquad a=\bigvee\{a'\prec a:a'\prec b\text{ or }a'\prec c,d\}.\]
\end{prp}

\begin{proof}
Assume $a\leq b\vee c$ and $a\leq b\vee d$.  Take $e$ such that $e\geq a'$ whenever $a'\prec a,b$ or $a'\prec a,c,d$.  We claim that then $e\geq a'$ whenever $a'\prec a,c$.  Indeed, for any $a'\prec a,c$, we have $a'\leq b\vee d$ so \eqref{precDecomposition} yields
\[a'=\bigvee\{a''\prec a':a''\prec b\text{ or }a''\prec d\}.\]
As $\{a''\prec a':a''\prec b\text{ or }a''\prec d\}\subseteq\{a''\prec a:a''\prec b\text{ or }a''\prec c,d\}$ and $e$ is above the latter set, we must have $e\geq a'$.  Thus $e\geq a'$ whenever $a'\prec a,b$ or $a'\prec a,c$.  But $a\leq b\vee c$ so by \eqref{precDecomposition} this implies $e\geq a$.  As $e$ was arbitrary,
\[a=\bigvee\{a'\prec a:a'\prec b\text{ or }a'\prec c,d\}.\qedhere\]
\end{proof}

\subsection{Hausdorff}\label{subsecHausdorff}

The name for the following is due to \eqref{OUN} in \autoref{Ubases}.

\begin{dfn}\label{Udef}
Define the \emph{Hausdorff} relation $\U$ on $S$ from $\prec$ as follows.
\[a\U b\qquad\Leftrightarrow\qquad\forall a'\prec a\ \ \forall b'\prec b\ \ \exists c\prec a,b\ \ \forall c'\prec a',b'\ (c'\prec c).\]
\end{dfn}

In a sense, $a\U b$ is saying that $a$ and $b$ have a meet relative to $\prec$ rather than $\leq$.

\begin{prp}\label{Uwedge}
If $\prec\ =\ \leq$ then $a\U b$ holds if and only if $a\wedge b$ exists.
\end{prp}

\begin{proof}
If $\prec\ =\ \leq$ then the right side of the definition of $a\U b$ reduces to the $a'=a$ and $b'=b$ case, in which case we must have $c=a\wedge b$ by the definition of a meet.
\end{proof}

Note $\leq$ is stronger than $\U$, as witnessed by taking $c=a'$ above, i.e.
\[\label{leqU}\tag{$\leq\ \subseteq\ \U$}a\leq b\qquad\Rightarrow\qquad a\U b.\]
In particular, $\U$ is reflexive.  We can also extend \eqref{Interpolation} to Hausdorff pairs.

\begin{prp}
\eqref{Interpolation} is equivalent to
\[\label{UInterpolation}\tag{$\U$-Interpolation}a\U b\qquad\Rightarrow\qquad\forall c\prec a,b\ \exists d\ (c\prec d\prec a,b).\]
\end{prp}

\begin{proof}
If \eqref{UInterpolation} holds then, as $\U$ is reflexive, the $a=b$ case reduces to \eqref{Interpolation}.  Conversely, assume \eqref{Interpolation} holds and $a\U b$.  For any $c\prec a,b$, \eqref{Interpolation} yields $a',b'\in S$ with $c\prec a'\prec a$ and $c\prec b'\prec b$.  Then $a\U b$ yields $d\prec a,b$ with $c\prec d$, as $c\prec a',b'$.
\end{proof}

We can then replace the second last $\prec$ with $\leq$ in \autoref{Udef}.

\begin{prp}
Under \eqref{LowerOrder}, \eqref{Interpolation} is equivalent to
\[\label{UEquivalent}\tag{$\U$-Equivalent}a\U b\quad\Leftrightarrow\quad\forall a'\prec a\ \ \forall b'\prec b\ \ \exists c\prec a,b\ \ \forall c'\leq a',b'\ (c'\prec c).\]
\end{prp}

\begin{proof}
As $\U$ is reflexive, \eqref{UEquivalent} with $a=b$ and $a'=b'=c'$ yields \eqref{Interpolation}.

Conversely, assume \eqref{LowerOrder} and \eqref{Interpolation} hold.  If $a'\prec a\U b\succ b'$ then we have $d\prec a,b$ with $d'\prec d$, for all $d'\prec a',b'$.  By \eqref{UInterpolation}, we have $c$ with $d\prec c\prec a,b$.  Now if $c'\leq a',b'$ then, for all $d'\prec c'$, \eqref{Auxiliarity} yields $d'\prec a',b'$ so $d'\prec d$.  As $d'$ was arbitrary,  \eqref{LowerOrder} yields $c'\leq d\prec c$ so \eqref{Auxiliarity} yields $c'\prec c$.  This verifies the $\Rightarrow$ part of \eqref{UEquivalent}.  The $\Leftarrow$ part is immediate from $\prec\ \subseteq\ \leq$ and the defintion of $\U$ in \autoref{Udef}.
\end{proof}

When $S$ also has meets, $\Rightarrow$ becomes $\Leftrightarrow$ in \eqref{UInterpolation}.  Put another way, $\U$ is really about $\prec$ preserving meets.

\begin{prp}
If $S$ is a $\wedge$-semilattice satisfying \eqref{Interpolation} then
\[\label{wedgePreservation}\tag{$\wedge$-Preservation}a\U b\qquad\Leftrightarrow\qquad\forall a'\prec a\ \forall b'\prec b\ (a'\wedge b'\prec a\wedge b).\]
\end{prp}

\begin{proof}
Assume $S$ is a $\wedge$-semilattice satisfying \eqref{Interpolation}.  If $a\U b$ then, for any $a'\prec a$ and $b'\prec b$, \eqref{Interpolation} yields $a'',b''\in S$ with $a'\prec a''\prec a$ and $b'\prec b''\prec b$.  By \eqref{Auxiliarity}, $a'\wedge b'\prec a'',b''$ so $a\U b$ yields $c\prec a,b$ with $a'\wedge b'\prec c$.  As $c\leq a\wedge b$, \eqref{Auxiliarity} again yields $a'\wedge b'\prec a\wedge b$.

Conversely, if the right side holds then, by \eqref{Interpolation}, we have some $c$ with $a'\wedge b'\prec c\prec a\wedge b$, which thus witnesses $a\U b$, again by \eqref{Auxiliarity}.
\end{proof}

Often $\U$ will hold for all bounded pairs (e.g. when $\prec$ is  the rather below relation as in \autoref{LocallyHausdorffprp} below) meaning that we have
\[\label{LocallyHausdorff}\tag{Locally Hausdorff}a,b\leq c\qquad\Rightarrow\qquad a\U b.\]
Then $\U$ will have the appropriate auxiliarity property.

\begin{prp}
If $S$ satisfies \eqref{LocallyHausdorff} and \eqref{Interpolation} then
\[\label{UAuxiliarity}\tag{$\U$-Auxiliarity}a\leq a'\U b'\geq b\qquad\Rightarrow\qquad a\U b.\]
\end{prp}

\begin{proof}
Say $a''\prec a\leq a'\U b'\geq b\succ b''$.  By \eqref{Auxiliarity}, $a''\prec a'\U b'\succ b''$ so we have $c'\prec a',b'$ with $d\prec c'$, for all $d\prec a'',b''$.  Then \eqref{UInterpolation} yields $c$ with $c'\prec c\prec a',b'$.  As $a,c\leq a'$, \eqref{LocallyHausdorff} implies $a''\prec a\U c\succ c'$, so we have $a''''\prec a,c$ with $d\prec a''''$, for all $d\prec a'',c'$.  Again \eqref{UInterpolation} yields $a'''$ with $a''''\prec a'''\prec a,c$.  Likewise, as $b,c\leq b'$, \eqref{LocallyHausdorff} implies $b''\prec b\U c\succ c'$, so we have $b''''\prec b,c$ with $d\prec b''''$, for all $d\prec b'',c'$.  Yet again \eqref{UInterpolation} yields $b'''$ with $b''''\prec b'''\prec b,c$.  As $a''',b'''\leq c$, \eqref{LocallyHausdorff} again implies $a''''\prec a'''\U b'''\succ b''''$, so we have $c'''\prec a''',b'''$ with $d\prec c'''$, for all $d\prec a'''',b''''$.  Thus $c'''\prec a'''\prec a$ and $c'''\prec b'''\prec b$.  Also, any $d\prec a'',b''$ satisfies $d\prec c'$ and hence $d\prec a'''',b''''$ so $d\prec c'''$.  Thus $c'''$ witnesses $a\U b$.
\end{proof}

\subsection{Disjoint}\label{subsecDisjoint}

We now wish to examine the disjoint relation $\perp$.  The appropriate definition of $\perp$ depends on whether the poset has a minimum or not.  For us, particularly when we deal with inverse semigroups in \autoref{InverseSemigroups}, it will be more natural to actually have a minimum so

\begin{center}
\textbf{from now on assume $S$ has a minimum $0$, i.e.}
\[\label{Minimum}\tag{Minimum}\forall a\in S\ (0\leq a).\]
\end{center}

The following concept will play an important role later on.

\begin{dfn}
We call $U\subseteq S$ \emph{$\prec$-round} if
\[\label{precRound}\tag{$\prec$-Round}\forall u\in U\ \exists v\in U\ (v\prec u).\]
\end{dfn}

For the moment we simply note that $S\setminus\{0\}$ is often $\prec$-round.

\begin{prp}\label{Snot0}
If \eqref{Approximation} holds then $S\setminus\{0\}$ is $\prec$-round.
\end{prp}

\begin{proof}
Assume \eqref{Approximation} holds but $S\setminus\{0\}$ is not $\prec$-round.  This means we have some non-zero $a\in S$ such that $b\prec a$ only holds possibly for $b=0$.  By \eqref{Approximation}, $a\leq0$ and hence $a=0$, a contradiction.
\end{proof}

\begin{dfn}
Define the \emph{disjoint} relation $\perp$ on $S$ by
\[a\perp b\qquad\Leftrightarrow\qquad a\wedge b=0.\]
\end{dfn}

We immediately see that $\perp$ has the following auxiliarity property.
\[\label{perpAuxiliarity}\tag{$\perp$-Auxiliarity}a\leq a'\perp b'\geq b\qquad\Rightarrow\quad a\perp b.\]
As long as $0\prec0$ then $\perp$ is also stronger than $\U$ from \autoref{Udef}, i.e.
\[\label{perpU}\tag{$\perp\ \subseteq\ \U$}a\perp b\qquad\Rightarrow\qquad a\U b.\]
In contrast to $\U$, however, $\perp$ is rarely reflexive.  Also, by \autoref{Snot0}, if \eqref{Approximation} holds then $\perp$ can be characterized by $\prec$, specifically
\[\label{perpEquivalent}\tag{$\perp$-Equivalent}a\perp b\qquad\Leftrightarrow\qquad\nexists c\neq0\ (c\prec a,b).\]

Let us also consider the following properties of $\perp$, whenever $b\vee c$ exists.
\begin{align}
\label{perpDecomposition}\tag{$\perp$-Decomposition}&a\perp b\quad\text{and}\quad a\leq b\vee c\hspace{-25pt}&&\Rightarrow\qquad a\leq c.\\
\label{perpDistributivity}\tag{$\perp$-Distributivity}&a\perp b\quad\text{and}\quad a\prec b\vee c\hspace{-25pt}&&\Rightarrow\qquad a\prec c.\\
\label{veePreservation}\tag{$\vee$-Preservation}&a\perp b\quad\text{and}\quad a\perp c\hspace{-25pt}&&\Rightarrow\qquad a\perp b\vee c.
\end{align}

\begin{prp}
We have the following implications.
\begin{align*}
\eqref{leqDecomposition}\qquad&\Rightarrow\qquad\eqref{perpDecomposition}\quad&&\Rightarrow\qquad\eqref{veePreservation}.\\
\eqref{precDistributivity}\qquad&\Rightarrow\qquad\eqref{perpDistributivity}.&&
\end{align*}
\end{prp}

\begin{proof}  If \eqref{leqDecomposition} holds then $a\leq b\vee c$ and $a\perp b$ implies
\[a=\bigvee\{a'\leq a:a'\leq b\text{ or }a'\leq c\}=\bigvee\{a'\leq a:a'\leq c\}\leq c,\]
so \eqref{perpDecomposition} holds.  If \eqref{perpDecomposition} holds then $a\perp b$ and $a\not\perp b\vee c$ implies we have non-zero $a'\leq a$ with $a'\leq a,b\vee c$.  Thus $a'\leq a\perp b$ so, by \eqref{perpDecomposition}, $a'\leq c$ and hence $a\not\perp c$, showing that \eqref{veePreservation} holds.

If \eqref{precDistributivity} holds and $a\prec b\vee c$ then \eqref{Shrinking} yields $b'\prec b$ and $c'\prec c$ with $a\leq b'\vee c'$.  Then $a\perp b\geq b'$ implies $a\leq c'\prec c$, by \eqref{perpDecomposition}, so $a\prec c$, by \eqref{Auxiliarity}, showing that \eqref{perpDistributivity} holds.
\end{proof}

%\begin{cor}
%If \eqref{precDistributivity} holds and $b\vee c$ exists then
%\[\label{perpDistributivity}\tag{$\perp$-Distributivity}a\perp b\quad\text{and}\quad a\prec b\vee c\qquad\Rightarrow\qquad a\prec c.\]
%\end{cor}
%
%\begin{proof}
%If $a\prec b\vee c$ then \eqref{Shrinking} yields $b'\prec b$ and $c'\prec c$ with $a\leq b'\vee c'$.  Then $a\perp b\geq b'$ implies $a\leq c'\prec c$, by \eqref{perpDecomposition}, so $a\prec c$, by \eqref{Auxiliarity}.
%\end{proof}

\subsection{Rather Below}\label{subsecRatherBelow}

\begin{dfn}\label{RatherBelowdfn}
Define the \emph{rather below} relation $\prec$ on $S$ by
\[a\prec b\qquad\Leftrightarrow\qquad a\leq b\quad\text{and}\quad\forall c\geq b\ \exists a'\perp a\ (\exists d\geq a',b\quad\text{and}\quad\forall d\geq a',b\ (c\leq d)).\]
\end{dfn}

This is a poset analog of `compact containment', as can be seen from \eqref{OprecN} below.

It is immediate from \autoref{RatherBelowdfn} that rather below $\prec$ is auxiliary to $\leq$, i.e.
\[\label{llAuxiliarity}\tag{Auxiliarity}a\leq a'\prec b'\leq b\qquad\Rightarrow\qquad a\prec b.\]
Also note that $0\prec 0$.  In the presence of joins, we can simplify the definition of $\prec$.

\begin{dfn}
$S$ is a \emph{conditional $\vee$-semilattice} if bounded pairs have joins, i.e.
\[a,b\leq c\qquad\Rightarrow\qquad a\vee b\text{ exists}.\]
\end{dfn}

When $S$ is a conditional $\vee$-semilattice, the rather below relation $\prec$ is given by
\[\label{RatherBelow}\tag{Rather Below}a\prec b\qquad\Leftrightarrow\qquad a\leq b\quad\text{and}\quad\forall c\geq b\ \exists a'\perp a\ (c\leq a'\vee b).\]
This form makes it clearer that $\prec$ is a variant of the rather below relation from \cite[Ch 5 \S5.2]{PicadoPultr2012} or the `well inside' relation from \cite[III.1.1]{Johnstone1986}.  It is analogous to the `way-below' relation $\ll$ from domain theory, which it is sometimes compared to.  For example, on the entire open set lattice of a compact Hausdorff space both the rather below and way-below relations coincide with compact containment $\overline{O}\subseteq N$.  In fact, as we will later note, the rather below relation recovers compact containment even on a basis which is just closed under finite unions.

However, the same can no longer be said for the way-below relation.

\begin{xpl}
Let $S$ be the clopen subsets of the Cantor space $C=2^\mathbb{N}$.  Certainly $\overline{C}\subseteq C$, even though $C\not\ll C$ in $S$.  To see this, pick any point $x\in C$.  The elements of $S$ avoiding $x$ form an ideal $I$ with $\bigcup I=C\setminus\{x\}$.  As $x$ is not isolated and the elements of $S$ are closed, $\bigvee I=C$ in $S$ even though $C\notin I$, i.e. $C\not\ll C$.
\end{xpl}

Just as the way-below relation is the strongest possible approximating relation \textendash\, see \cite[Proposition I-1.15]{GierzHofmannKeimelLawsonMisloveScott2003} \textendash\, the rather below relation is the strongest possible $\succ$-round $\perp$-distributive relation, at least under appropriate conditions.

Note $\succ$-round refers to the dual to \eqref{precRound}, i.e. $U$ is $\succ$-round if
\[\label{succRound}\tag{$\succ$-Round}\forall u\in U\ \exists v\in U\ (u\prec v).\]

%\begin{prp}
%If $S$ is a $\vee$-semilattice and \eqref{perpDecomposition} holds then $\prec_R$ is the intersection of all auxiliary $\prec\ \subseteq\ \leq$ satisfying \eqref{succRound} and \eqref{perpDistributivity}.
%\end{prp}

%If $S$ is $<$-round and $<$-distributive then $<$ is weaker than the rather below relation $\prec$, i.e.
%\[a\prec b\qquad\Rightarrow\qquad a<b.\]
\begin{prp}\label{precMinimal}
If $S$ is a conditional $\vee$-semilattice, the rather below relation is contained in all auxiliary $\prec\ \subseteq\ \leq$ satisfying \eqref{succRound} and \eqref{perpDistributivity}.  If $S$ is even a $\vee$-semilattice satisfying \eqref{perpDecomposition} then the rather below relation is precisely the intersection of all such $\prec$.
\end{prp}

\begin{proof}
Say $a$ is rather below $b$.  As $S$ satisfies \eqref{succRound}, we have $c\succ b$.  In particular, $b\leq c$ so we must have $a'\perp a$ with $c\leq a'\vee b$.  As $a\leq b\prec c\leq a'\vee b$, \eqref{perpDistributivity} yields $a\prec b$.

Now say $S$ is a $\vee$-semilattice satisfying \eqref{perpDecomposition} and $a$ is not rather below $b$.  We need to find auxiliary $\prec\ \subseteq\ \leq$ with $a\not\prec b$ satisfying \eqref{succRound} and \eqref{perpDistributivity}.  If $a\nleq b$ then we can just take $\prec\ =\ \leq$.  If $a\leq b$ then we must have $c\geq b$ such that there is no $a'\perp a$ with $c\leq a'\vee b$.  In this case, define $\prec$ by
\[d\prec e\qquad\Leftrightarrow\qquad d\leq e\text{ and }\exists d'\perp d\ (c\leq d'\vee e).\]
Certainly $\prec\ \subseteq\ \leq$ is auxiliary to $\leq$ and $a\not\prec b$.  Also we always have $d\prec c\vee d$ (witnessed by $d'=0$) so \eqref{precRound} holds.  To see that \eqref{perpDistributivity} holds, say $d\perp e$ and $d\prec e\vee f$, so we have $d'\perp d$ with $c\leq d'\vee e\vee f$.  By \eqref{veePreservation}, $d\perp d'\vee e$ so it follows that $d\prec f$, as required.
\end{proof}

Let us now stop considering arbitrary auxiliary relations and restrict our attention to the rather below relation.
\begin{center}
\textbf{From now on $\prec$ specifically refers to the rather below relation.}
\end{center}

Let us now show that the $a\leq b$ part of \eqref{RatherBelow} follows automatically if $a$ and $b$ are bounded, as long as we have some degree of distributivity.

\begin{prp}\label{llequivalent}
If $S$ is a conditional $\vee$-semilattice satisfying \eqref{perpDecomposition},
\[\label{allb}a\prec b\qquad\Leftrightarrow\qquad \exists c\geq a,b\quad\text{and}\quad\forall c\geq b\ \exists a'\perp a\ (c\leq a'\vee b).\]
\end{prp}

\begin{proof}
We immediately have $\Rightarrow$.  Conversely, assume the right side holds.  Then we can take $c\geq a,b$ and $a'\perp a$ with $a\leq c\leq a'\vee b$.  Then \eqref{perpDecomposition} yields $a\leq b$ and hence $a\prec b$, by \eqref{RatherBelow}.
\end{proof}

Another thing worth pointing out is that in general the inequality $c\leq a'\vee b$ above can be strict.  Although if $S$ were $\leq$-distributive then we could replace $a'$ with something smaller to obtain $c=a'\vee b$.  Even under $\prec$-distributivity, we can still ensure that $a'\vee b$ is not too large.

\begin{prp}
If $S$ is a $\prec$-distributive conditional $\vee$-semilattice then
\[\label{Complements}\tag{Complements}a\prec b\leq c\prec d\qquad\Rightarrow\qquad\ \exists a'\perp a\ (c\prec a'\vee b\prec d).\]
\end{prp}

\begin{proof}
Assume $a\prec b\leq c\prec d$.  By \eqref{Interpolation}, we have $d'$ with $c\prec d'\prec d$.  Now take $a'\perp a$ with $d'\leq a'\vee b$.  As $c\prec d'\leq a'\vee b$, $\prec$-distributivity yields $a''\prec a'$ and $b'\prec b$ with $c\prec a''\vee b'\prec d'$.  As $b'\prec b\leq c\prec d'$, we have
\[c\prec a''\vee b'\leq a''\vee b\leq d'\prec d.\]
%\[c\prec a''\vee b'\vee b=a''\vee b\leq d'\prec d.\]
As $a''\prec a'\perp a$, we also have $a''\perp a$ so we are done.
\end{proof}

\begin{cor}
If $S$ is a $\prec$-distributive conditional $\vee$-semilattice then the rather below relation defined within
\[d^\succ=\{a\in S:a\prec d\}\]
is just the restriction to $d^\succ$ of the rather below relation $\prec$ defined within $S$.
\end{cor}

\begin{proof}
From \eqref{Complements} it follows that $a\prec b\prec d$ implies that $a$ is rather below $b$ in $d^\succ$.  Conversely, the restriction of $\prec$ to $d^\succ$ is \eqref{succRound} in $d^\succ$, by \eqref{Interpolation}, and also remains $\prec$-Distributive and hence $\perp$-Distributive in $d^\succ$.  Thus $\prec$ must contain the rather below relation in $d^\succ$, by \autoref{precMinimal}.
\end{proof}

Note we are now considering $\U$ to be defined specifically from the rather below relation too, rather than some arbitrary auxiliary relation.  This allows us to prove further properties of $\U$.  For example, under \eqref{precDistributivity}, $\U$ not only contains $\leq$ but in fact holds for all bounded pairs.

\begin{prp}\label{LocallyHausdorffprp}
If $S$ is a $\prec$-distributive conditional $\vee$-semilattice then
\[\tag{\ref{LocallyHausdorff}}a,b\leq c\qquad\Rightarrow\qquad a\U b.\]
\end{prp}

\begin{proof}
Say $a'\prec a$ and $b'\prec b$, for some $a,b\leq c$.  By \eqref{RatherBelow}, $a'\prec a$ means we have $d\perp a'$ with $c\leq d\vee a$.  As $b'\prec b\leq c\leq d\vee a$, \eqref{precDistributivity} yields $d'\prec d$ and $a''\prec a$ with $b'\prec d'\vee a''\prec b$.  Note $a''\prec a,b$, by \eqref{Auxiliarity}.  Also, for any $c'\prec a',b'$, we have $c'\prec a'\perp d\succ d'$ and $c'\prec b'\prec d'\vee a''$ and hence $c'\prec a''$, by \eqref{perpDistributivity}.  Thus $a''$ witnesses $a\U b$. 
\end{proof}

If $S$ is even a lattice then $\U$ is trivial, i.e. $\prec$ preserves meets.

\begin{prp}
If $S$ is a distributive lattice, \eqref{Interpolation} implies $\U\ =S\times S$.
\end{prp}

\begin{proof}
By \eqref{wedgePreservation}, we have to show that, for all $a,b,c\in S$,
\[a\prec b,c\qquad\Rightarrow\qquad a\prec b\wedge c.\]
If $a\prec b,c$ and $d\geq b\wedge c$ then, as $S$ is a lattice, we can set $d'=d\vee b\vee c$.  Applying \eqref{RatherBelow} to $a\prec b$ and then $a\prec c$, we obtain $b',c'\perp a$ with $d'\leq b'\vee b\leq c'\vee c$.  By \eqref{veePreservation}, $b'\vee c'\perp a$ and, by distributivity,
\[d\leq d'\leq(b'\vee b)\wedge(c'\vee c)\leq b'\vee c'\vee(b\wedge c).\]
As $d$ was arbitrary, this shows that $a\prec b\wedge c$.
\end{proof}

A similar argument shows $\prec$ preserves joins in a distributive lattice.  In fact, $\prec$-distributivity suffices, so long as $S$ is also $\succ$-round.  We examine such posets further in the next section, but first let us show that \eqref{succRound} (for $S$) is not automatic from the other conditions we have considered so far.

\begin{xpl}
Take a sequence $(a_n)$ of subsets of $\mathbb{N}$ such that $a_1=\emptyset$ and, for all $n\in\mathbb{N}$, $a_n\subseteq a_{n+1}$ and $a_{n+1}\setminus a_n$ is infinite (e.g. $a_n=\mathbb{N}\setminus\{k2^{n-1}:k\in\mathbb{N}\}$).  Let $S$ be the collection of subsets of $\mathbb{N}$ with finite symmetric difference with some $a_n$, i.e.
\[S=\{b\subseteq\mathbb{N}:(b\setminus a_n)\cup(a_n\setminus b)\text{ is finite, for some }n\in\mathbb{N}.\}\]
Considering $S$ as a poset w.r.t. inclusion $\subseteq$, we immediately see that $S$ is a distributive lattice with minimum $\emptyset$ and meets and joins given by $\cap$ and $\cup$.  Moreover, $\prec$ is just the restriction of $\subseteq$ to finite subsets on the left, i.e.
\begin{equation}\label{precFinite}
a\prec b\qquad\Leftrightarrow\qquad a\text{ is a finite subset of }b,
\end{equation}
for all $a,b\in S$.  Indeed, if $a$ is finite then, for any $c\supseteq a$, we have $a'=c\setminus a\in S$ so $a'\perp a$ and $c=a'\cup a$, i.e. $a\prec a$ and hence $a\prec b$, for any $b\supseteq a$.  On the other hand, if $a$ is infinite then any $a'\in S$ with $a'\perp a$ must be finite.  Thus, for any $b\supseteq a$, we can take $c\in S$ with $c\supseteq b$ and $c\setminus b$ infinite, which means there is no $a'\in S$ with $a'\perp a$ and $c\subseteq a'\cup b$.  This shows that there is no $b\in S$ with $a\prec b$, i.e. \eqref{succRound} fails for infinite $a$.  However, from \eqref{precFinite} it can also be verified that $S$ satisfies \eqref{precDistributivity} (note \eqref{LowerOrder} and hence the $\Leftarrow$ part of \eqref{precDistributivity} would fail if just considered the poset formed from $(a_n)$ without adding in finite differences).
\end{xpl}

\section{Basic Posets}\label{secBasicPosets}

We are now in a position to examine posets we are primarily interested in.

\begin{dfn}\label{BasicUPoset}
We call $S$ a \emph{$\U$-basic} poset if $S$ is a $\prec$-distributive and
\[\label{HausdorffJoins}\tag{$\U$-Joins}a\vee b\text{ exists}\qquad\Leftrightarrow\qquad\exists a',b'\ (a\prec a'\U b'\succ b),\]
\end{dfn}

Just to reiterate, $\prec$ denotes the rather below relation defined from $\leq$ as in \autoref{RatherBelowdfn} and $\U$ denotes the Hausdorff relation defined from $\prec$ as in \autoref{Udef}.

Our goal will be to show that $\U$-basic posets are dual to `$\cup$-bases' (see \autoref{subsecUBases}) of locally compact locally Hausdorff topological spaces, hence the name `$\U$-basic'.  In fact, we will show in \autoref{SpacesFromPosets} that the following slightly more general posets can also be represented as bases of locally compact locally Hausdorff spaces.

\begin{dfn}\label{BasicVeePoset}
We call $S$ a \emph{basic} poset if
\begin{enumerate}
\item $S$ is $\succ$-round,
\item $S$ is $\prec$-distributive, and
\item $S$ is a conditional $\vee$-semilattice.
\end{enumerate}
\end{dfn}

\begin{prp}\label{Ubasic=>basic}
Every $\U$-basic poset $S$ is a basic poset.
\end{prp}

\begin{proof}
First note that the $\Rightarrow$ part of \eqref{HausdorffJoins}, even just in the $a=b$ case, is equivalent to \eqref{succRound}.  Thus whenever $a,b\leq c$, we have some $c'\succ c$ which means $a\prec c'\U c'\succ b$ and hence $a\vee b$ exists, by the $\Leftarrow$ part of \eqref{HausdorffJoins}.  Thus $S$ is also a conditional $\vee$-semilattice and hence a basic poset.
\end{proof}

\subsection{Basic Properties}\label{subsecBasicProperties}

\begin{prp}\label{aveebllc}
If $S$ is a basic poset then
\[\label{Predomain}\tag{Predomain}a,b\prec c\qquad\Rightarrow\qquad a\vee b\prec c.\]
\end{prp}

\begin{proof}
Take $a,b\prec c$.  As $S$ is $\succ$-round, for any $d\geq c$, we have $e\succ d$.  Then we can take $a'\perp a$ with $e\leq a'\vee c$.  As $d\prec e\leq a'\vee c$, \eqref{Shrinking} yields $a''\prec a'$ with $d\leq a''\vee c$.  Likewise, we can take $b'\perp b$ with $d\prec e\leq a'\vee c\leq b'\vee c=f$ so \eqref{Shrinking} again yields $b''\prec b'$ with $d\leq b''\vee c$.  As $a',b'\leq f$, \eqref{LocallyHausdorff} implies $a''\prec a'\U b'\succ b''$ which yields $c'\prec a',b'$ with $c'\succ c''$, for all $c''\leq a'',b''$.  As $c'\prec a'\perp a$ and $c'\prec b'\perp b$, \eqref{veePreservation} yields $c'\perp a\vee b$.  As $d\leq a''\vee c,b''\vee c$, \eqref{leqDecomposition} and \autoref{veeoverwedge} yields
\[d=\bigvee\{c''\leq d:c''\leq c\text{ or }c''\leq a'',b''\}\leq c\vee c'.\]
As $d$ was arbitrary, $a\vee b\prec c$.
\end{proof}

In fact, as basic posets satisfy \eqref{precDistributivity} and hence \eqref{Interpolation},
\[a,b\prec c\qquad\Rightarrow\qquad\exists d\ (a,b\prec d\prec c).\]
This means basic posets are `(stratified) predomains' in the sense of \cite{Keimel2016} or `abstract bases' in the sense of \cite[Lemma 5.1.32]{Goubault2013}.

Generally, to verify $a\prec b$ we need to consider all $c\geq b$.  However, for basic posets, it suffices to consider just one $c\succ a$.

\begin{prp}\label{llexistsUSS}
If $S$ is a basic poset then
\[a\prec b\qquad\Leftrightarrow\qquad\exists c\succ a\ \exists a'\perp a\ (c\leq a'\vee b).\]
\end{prp}

\begin{proof}
The $\Rightarrow$ part is immediate from \eqref{succRound}.  Conversely, if $a'\perp a\prec c\leq a'\vee b$ then \eqref{perpDistributivity} yields $a\prec b$.
\end{proof}

This is possibly a good point to summarize some of the most important properties we have considered so far.  Specifically, recall that any basic poset satisfies
\begin{align}
\tag{\ref{Auxiliarity}}a\leq b\prec c\leq d\qquad&\Rightarrow\qquad a\prec d.\\
\tag{\ref{Approximation}}\forall c\prec a\ (c\leq b)\qquad&\Rightarrow\qquad a\leq b.\\
\tag{\ref{UAuxiliarity}}a\leq a'\U b'\geq b\qquad&\Rightarrow\qquad a\U b.\\
\tag{\ref{UInterpolation}}c\prec a,b\text{ and }a\U b\qquad&\Rightarrow\qquad\exists d\ (c\prec d\prec a,b).\\
\tag{\ref{Complements}}a\prec b\leq c\prec d\qquad&\Rightarrow\qquad\ \exists a'\perp a\ (c\prec a'\vee b\prec d).\\
\tag{\ref{Predomain}}a,b\prec c\qquad&\Rightarrow\qquad a\vee b\prec c.
\end{align}

\subsection{Boolean Examples}\label{subsecBooleanExamples}

The construction of locally compact locally Hausdorff topological spaces from basic posets will follow the construction of Stone spaces from Boolean algebras.  This will come in the next section on filters, although the potential for doing this is already hinted at by the following.

\begin{prp}
Boolean algebras are just bounded basic posets with $\prec\ =\ \leq$.
\end{prp}

\begin{proof}
If $S$ is a Boolean algebra then by definition $S$ is a distributive complemented lattice.  In particular, $S$ is bounded, i.e. $S$ has a maximum (for complements to even be defined) and the existence of complements implies that $\prec$ is $\leq$, so $S$ is certainly $\prec$-round, i.e. $S$ is a bounded basic poset.

Conversely, if $S$ is a bounded basic poset with $\prec\ =\ \leq$ then, in particular, $S$ is not just a conditional $\vee$-semilattice but a true $\vee$-semilattice.  Also $\prec\ =\ \leq$ implies $\prec$ is reflexive so every element has a complement and hence $S$ is also a $\wedge$-semilattice, as $a\wedge b=(a^c\vee b^c)^c$.  Also $\prec$-distributivity becomes $\leq$-distributivity so $S$ is a distributive complemented lattice, i.e. a Boolean algebra.
\end{proof}

\begin{prp}
Generalized Boolean algebras are just $\U$-basic posets with
\[\prec\ =\ \leq\quad\text{and}\quad\U\ =S\times S.\]
\end{prp}

\begin{proof}
If $S$ is a $\U$-basic poset with $\U\ =S\times S$ then, for any $a,b\in S$, $\succ$-roundness yields $a',b'\in S$ such that $a\prec a'\U b'\succ b$.  So $a\vee b$ exists, by \eqref{HausdorffJoins}, showing that $S$ is a true $\vee$-semilattice.  If $\prec\ =\ \leq$ too then $\prec$ is reflexive so, whenever $a\leq b$, we have $a\prec a\leq b\prec b$.  Then \eqref{Complements} yields $a'\perp a$ with $b=a'\vee a$, i.e. $a'=b\setminus a$ is a complement of $a$ relative to $b$.  It follows that $S$ is also a $\wedge$-semilattice, as $a\wedge b=a\vee b\setminus((a\vee b\setminus a)\vee(a\vee b\setminus b))$.  Thus $S$ is a distributive relatively complemented lattice, i.e. a generalized Boolean algebra.

Conversely, if $S$ is a generalized Boolean algebra then the existence of relative complements yields $\prec\ =\ \leq$.  Thus $S$ is $\prec$-round and $\prec$-distributive.  Also $\prec\ =\ \leq$ means that $a\U b$ is just saying $a$ and $b$ have a meet, by \autoref{Uwedge}.  As $S$ is a lattice this is always true so $\U\ = S\times S$.
\end{proof}

Dropping the $\U\ =S\times S$ condition yields a further `locally Hausdorff' generalization of Boolean algebras.  With this in mind, we make the following definition.

\begin{dfn}
We call $S$ \emph{locally Boolean} or a \emph{local Boolean algebra} if
\[\downarrow a=\{b\in S:b\leq a\}\]
is a Boolean algebra, for all $a\in S$.
\end{dfn}

In particular, a local Boolean algebra has meets whenever it has joins, i.e.
\[a\vee b\text{ exists}\qquad\Rightarrow\qquad a\wedge b\text{ exists}.\]
More interesting is the converse, when $S$ has as many joins as possible.

\begin{prp}
A local Boolean algebra is just a basic poset with $\prec\ =\ \leq$.  In this case, $S$ is a $\U$-basic poset if and only if it has joins whenever it has meets, i.e.
\[\label{wedgeJoins}\tag{$\wedge$-Joins}a\wedge b\text{ exists}\qquad\Rightarrow\qquad a\vee b\text{ exists}.\]
\end{prp}

\begin{proof}
If $S$ is locally Boolean then certainly $S$ is a conditional $\vee$-semilattice and again the existence of relative complements yields $\prec\ =\ \leq$.  Thus $S$ is $\prec$-round and $\prec$-distributive and hence a basic poset.  The converse follows from \eqref{Complements} as above.  For the last part note again that $\prec\ =\ \leq$ means that $a\U b$ is just saying $a$ and $b$ have a meet, by \autoref{Uwedge}.  So if $S$ is also a $\U$-basic poset and $a\wedge b$ exists then $a\prec a\U b\succ b$ so $a\vee b$ exists, by \eqref{HausdorffJoins}.  Conversely, if joins exist whenever meets do and $a\prec a'\U b'\succ b$ then \eqref{UAuxiliarity} yields $a\U b$, i.e. $a\wedge b$ exists so $a\vee b$ exists, showing that \eqref{HausdorffJoins} holds.
\end{proof}

\begin{xpl}\label{bug-eyedC}
For an example of a local Boolean algebra satisfying \eqref{wedgeJoins} which is not a generalized Boolean algebra, consider the compact open Hausdorff subsets $S$ of the `bug-eyed' Cantor space $C$, i.e. the Cantor space $\{0,1\}^\mathbb{N}$ with an extra copy $x'$ of one point $x$.  More precisely, let $C$ be the quotient space obtained from two copies $\{0,1\}^\mathbb{N}\sqcup\{0,1\}^\mathbb{N}{'}$ of the Cantor space by identifying each $c$ and $c'$ except when $c=x$.  Then $C\setminus\{x\}$ and $C\setminus\{x'\}$ are maximal elements of $S$, necessarily with no meet (which can also be seen directly from the fact their intersection $C\setminus\{x,x'\}$ is not compact).
\end{xpl}

For examples of $\U$-basic posets where $\prec\ \neq\ \leq$, we turn to $\cup$-bases of topological spaces.  But first we need to examine the compact containment relation.

\subsection{Compact Containment}\label{subsecCompactContainment}

\[\textbf{Throughout this section assume $G$ is a topological space}.\]

While topological spaces are usually denoted by letters like $X$, we use $G$ to emphasize that we are primarily interested in spaces coming from \'{e}tale groupoids \textendash\, see \autoref{Groupoids}.

\begin{dfn}\label{CompactContainment}
The \emph{compact containment} relation $\Subset$ on subsets of $G$ is given by
\[O\Subset N\qquad\Leftrightarrow\qquad\exists\text{ compact }C\ (O\subseteq C\subseteq N).\]
\end{dfn}
We immediately see that compact containment is auxiliary to containment, i.e.
\[O\subseteq O'\Subset N'\subseteq N\qquad\Rightarrow\qquad O\Subset N.\]
Compact containment is also related to closures, particularly for Hausdorff subsets.

\begin{prp}
For any Hausdorff $O,N,M\subseteq G$ with $O\subseteq N\Subset M$,
\begin{equation}\label{OSubNequivs}
O\Subset N\qquad\Leftrightarrow\qquad \overline{O}\cap N\text{ is compact}\qquad\Leftrightarrow\qquad\overline{O}\cap M\subseteq N.
\end{equation}
\end{prp}

\begin{proof}
Assume $O\Subset N$, i.e. $O\subseteq C\subseteq N$, for some compact $C$.  As $N$ is Hausdorff, $C$ must be relatively closed in $N$ so $\overline{O}\cap N\subseteq C$.  As closed subsets of compact spaces are compact, $\overline{O}\cap N$ is also compact.  Conversely, if $\overline{O}\cap N$ is compact then $O\Subset N$, as $O\subseteq\overline{O}\cap N\subseteq N$.

Again assume $O\Subset N$, i.e. we have compact $C$ with $O\subseteq C\subseteq N\subseteq M$.  As $M$ is Hausdorff, $C$ must be relatively closed in $M$ so $\overline{O}\cap M\subseteq C\subseteq N$.  Conversely, assume $\overline{O}\cap M\subseteq N$.  As $N\Subset M$, we have compact $C$ with $O\subseteq N\subseteq C\subseteq M$.  Again $C$ is relatively closed in $M$, as $M$ is Hausdorff, so $\overline{O}\cap N=\overline{O}\cap M\subseteq C$ is compact, as a closed subset of $C$.
\end{proof}

We call $G$ \emph{locally compact} if every point $g\in G$ has a neighbourhood base of compact subsets.  In other words, $G$ is locally compact if each neighbourhood filter
\[N_g=\{N\subseteq G:g\in N^\circ\}\text{ is $\Subset$-round}.\]

\begin{rmk}
In Hausdorff spaces it is common to use weaker notions of local compactness like in \autoref{LCLHchars} \eqref{L(CH)} below.  However, the stronger notion given here is standard for more general potentially non-Hausdorff spaces \textendash\, see \cite[Definition O-5.9]{GierzHofmannKeimelLawsonMisloveScott2003} or \cite[Definition 4.8.1]{Goubault2013}.
\end{rmk}

\begin{prp}\label{GHausdorffChars}
If $G$ is locally compact and $T_0$ with basis $S$, TFAE.
\begin{enumerate}
\item\label{GHausdorff} $G$ is Hausdorff.
\item\label{Compact=Closed} Every compact subset of $G$ is closed.
\item\label{Subset=>ClosedContainment} For every $O,N\in S$,
\[O\Subset N\qquad\Rightarrow\qquad\overline{O}\subseteq N.\]
\end{enumerate}
\end{prp}

\begin{proof}\
\begin{itemize}
\item[\eqref{GHausdorff}$\Rightarrow$\eqref{Compact=Closed}] Standard, even without local compactness.

\item[\eqref{Compact=Closed}$\Rightarrow$\eqref{Subset=>ClosedContainment}] If $O\subseteq C\subseteq N$, for compact $C$, then \eqref{Compact=Closed} implies $C$ is closed so $\overline{O}\subseteq C\subseteq N$.

\item[\eqref{Subset=>ClosedContainment}$\Rightarrow$\eqref{GHausdorff}] Say \eqref{GHausdorff} fails, i.e. $G$ is not Hausdorff, so we have $g,h\in G$ whose neighbourhoods are never disjoint.  As $G$ is $T_0$, we can assume w.l.o.g. that we have some $N\in S$ with $g\notin N\ni h$.  As $G$ is locally compact, we have $O\in S$ with $h\in O\Subset N$.  By assumption, every neighbourhood of $g$ intersects $O$ so $g\in\overline{O}\setminus N$, i.e. $\overline{O}\not\subseteq N$ so \eqref{Subset=>ClosedContainment} fails.\qedhere
\end{itemize}
\end{proof}

Now we want to relate $\Subset$ with $\prec$ on suitable bases $S$ of $G$, taking inclusion $\subseteq$ as the order relation $\leq$ which in turn defines the rather below relation $\prec$.  We can note immediately that arbitrary bases will not do, e.g. any maximal $O\in S$ will trivially satisfy $O\prec O$, even when $O$ is not compact, i.e. $O\not\Subset O$.  We avoid this problem by dealing with $\Supset$-round bases that are closed under bounded finite unions.

\begin{prp}
If $S$ is a basis of $G$ that is closed under bounded finite unions (i.e. $O\cup N\in S$ whenever $O,N\subseteq M$ and $O,N,M\in S$) then, for any $O,N,M\in S$,
\begin{equation}\label{OprecNSubM}
O\prec N\Subset M\qquad\Rightarrow\qquad O\subseteq N\Subset M\text{ and }\overline{O}\cap M\subseteq N\qquad\Rightarrow\qquad O\Subset N.
\end{equation}
If $G$ is also $\Supset$-round then, for any $O,N\in S$,
\begin{equation}\label{OprecN=>}
O\prec N\qquad\Leftrightarrow\qquad O\subseteq N\text{ and }\bigcup_{N\Subset M\in S}\overline{O}\cap M\subseteq N\qquad\Rightarrow\qquad O\Subset N.
\end{equation}
%\[O\prec N\qquad\Leftrightarrow\qquad O\subseteq N\text{ and }\overline{O}\cap\bigcup\{M\supseteq N:M\in S\}\subseteq N.\]
\end{prp}

\begin{proof}
Assume $O,N,M\in S$ and $O\prec N\Subset M$.  Then certainly $O\leq N$, i.e. $O\subseteq N$, and we must have compact $C$ with $N\subseteq C\subseteq M$ as well as $L\in S$ with $O\cap L=\emptyset$ and $M\subseteq N\cup L$ (as $S$ is closed under bounded finite unions, $\vee$ is $\cup$).  Thus $\overline{O}\cap L=\emptyset$ and hence $O\subseteq\overline{O}\cap C\subseteq\overline{O}\cap M\subseteq N$.  Also, as a closed subset of a compact space, $\overline{O}\cap C$ is compact and hence $O\Subset N$, thus proving \eqref{OprecNSubM}.

Now assume that $G$ is also $\Supset$-round.  The $\Rightarrow$ parts of \eqref{OprecNSubM} follow immediately from \eqref{OprecN=>}.  Conversely, for the first $\Leftarrow$ in \eqref{OprecNSubM}, assume $O\subseteq N$ and 
\begin{equation}
\label{overOcapM}\bigcup_{N\Subset M\in S}\overline{O}\cap M\subseteq N.
\end{equation}
Given $M\in S$ with $N\subseteq M$, we can take $P\in S$ with $M\Subset P$, as $S$ is $\Supset$-round, so we have some compact Hausdorff $C$ with $M\subseteq C\subseteq P$.  By \eqref{overOcapM}, $\overline{O}\cap C\subseteq\overline{O}\cap P\subseteq N$ so $C\setminus N$ is disjoint from $\overline{O}$.  As $C\setminus N$ is relatively closed in $C$ it is also compact.  Thus, for each $g\in C\setminus N\subseteq P$, we have some $O_g\in S$ with $g\in O_g\subseteq P\setminus\overline{O}$.  So the $(O_g)$ cover $C\setminus N$ and we must have some finite subcover $O_{g_1},\ldots,O_{g_n}$.  As $M\subseteq C$, they also cover $M\setminus N$, and as they are all contained in $P$, their union $L$ must be in $S$.  As $L$ covers $M\setminus N$, we have $M\subseteq N\cup L$ as well as $O\cap L=\emptyset$.  As $M$ was arbitrary, this shows that $O\prec N$.
\end{proof}

\begin{xpl}
To see that $\Subset$ can be strictly weaker than $\prec$ above, we can again consider the bug-eyed Cantor space $C$ from \autoref{bug-eyedC}, this time taking $S$ to be the collection of all (even non-Hausdorff) compact open subsets.  Thus $S$ is closed under arbitrary finite unions and is also certainly $\Supset$-round, as $O\Subset O\Subset C\in S$, for all $O\in S$.  In particular, $O\Subset O$ for $O=C\setminus\{x'\}\in S$, even though $\overline{O}\cap C=C\not\subseteq O$ and hence $O\not\prec O$.
\end{xpl}

%For an example of a local Boolean algebra satisfying \eqref{wedgeJoins} which is not a generalized Boolean algebra, consider the compact open Hausdorff subsets $S$ of the `bug-eyed' Cantor space $C$, i.e. the Cantor space $\{0,1\}^\mathbb{N}$ with an extra copy $x'$ of one point $x$.  More precisely, let $C$ be the quotient space obtained from two copies $\{0,1\}^\mathbb{N}\sqcup\{0,1\}^\mathbb{N}{'}$ of the Cantor space by identifying each $c$ and $c'$ except when $c=x$.  Then $C\setminus\{x\}$ and $C\setminus\{x'\}$ are maximal elements of $S$, necessarily with no meet (which can also be seen directly from the fact their intersection $C\setminus\{x,x'\}$ is not compact).

\subsection{\texorpdfstring{$\cup$-Bases}{U-Bases}}\label{subsecUBases}

\[\textbf{Throughout this section again assume $G$ is a topological space}.\]

\begin{dfn}\label{Ubasisdef}
We call a basis $S\ni\emptyset$ of $G$ a \emph{$\cup$-basis} if, for all $O,N\in S$,
\[\label{UnionClosed}\tag{$\cup$}O\cup N\in S\qquad\Leftrightarrow\qquad O\cup N\subseteq C,\text{ for some compact Hausdorff }C\subseteq G.\]
\end{dfn}

In particular, taking $O=N$ in the $\Rightarrow$ part of \eqref{UnionClosed}, we see that every element of a $\cup$-basis is required to be contained in some compact Hausdorff subset.  Then the $\Leftarrow$ part of \eqref{UnionClosed} ensures that $S$ is a conditional $\vee$-semilattice.

For Hausdorff $G$, $\cup$-bases are just bases which consist of relatively compact open subsets and which are closed under taking finite unions.  Note this union condition is crucial if we hope to recover the space from the order structure of the basis, as arbitrary bases may fail to distinguish spaces like the unit interval from the Cantor space.  Indeed, any second countable compact Hausdorff space has a countable basis of regular open sets.  This generates a countable Boolean algebra of regular open sets, which is atomless as long as the space has no isolated points.  However, all countable atomless Boolean algebras are isomorphic.

We call $G$ \emph{locally Hausdorff} if every point has a Hausdorff neighbourhood.  Note points are, in particular, compact Hausdorff.  In locally Hausdorff spaces, it turns out that all compact Hausdorff subsets have Hausdorff neighbourhoods.  In particular, we can always extend $C$ in \eqref{UnionClosed} to some open Hausdorff subset to obtain
\[\label{UnionClosedEquiv}\tag{$\cup$}O\cup N\in S\qquad\Leftrightarrow\qquad O\cup N\Subset M,\text{ for some open Hausdorff }M\subseteq G.\]
%By \autoref{LHCneighbourhoods}, we can also restate \eqref{UnionClosed} for $\cup$-bases $S$ using $\Subset$ as

\begin{prp}\label{LHCneighbourhoods}
If $G$ is locally Hausdorff then every compact Hausdorff subset $C\subseteq G$ is contained in some open Hausdorff subset.
\end{prp}

\begin{proof}
For every $g\in C$, take some Hausdorff open $O_g\ni g$.  As $C$ is Hausdorff, for every $h\in C\setminus O_g$, we have disjoint open $N_h\ni g$ and $M_h\ni h$.  As $C$ is compact and $O_g$ is open, $C\setminus O_g$ is compact so we have some finite subcover $M'_g=M_{h_1}\cup\ldots\cup M_{h_k}$ of $C\setminus O_g$.  Let $N'_g=O_g\cap N_{h_1}\cap\ldots\cap N_{h_k}$.  Again as $C$ is compact, we have some finite subcover $N=N'_{g_1}\cup\ldots\cup N'_{g_j}$ of $C$.  Let
\[O=N\cap(O_{g_1}\cup M'_{g_1})\cap\ldots\cap(O_{g_j}\cap M'_{g_j}).\]
Certainly $O$ is an open subset containing $C$.  To see that $O$ is Hausdorff, take any $g,h\in O$.  As $g\in N$, we have $g\in N'_{g_l}$ for some $l$.  We also have $h\in O_{g_l}\cup M'_{g_l}$.  If $h\in O_{g_l}$ then, as $g\in N'_{g_l}\subseteq O_{g_l}$ too and $O_{g_l}$ is Hausdorff, we can separate $h$ and $g$ with disjoint open sets.  If $h\in M'_{g_l}$ then, as $N'_{g_l}\cap M'_{g_l}=\emptyset$, we already have disjoint open subsets separating $h$ and $g$.  As $g$ and $h$ were arbitrary, we are done.
\end{proof}

\begin{prp}\label{LCLHchars}
The following are equivalent.
\begin{enumerate}
\item\label{Ubasis} $G$ has a $\cup$-basis.
\item\label{LCLH} $G$ is locally compact locally Hausdorff.
\item\label{L(CH)} Every point $g\in G$ has a compact Hausdorff neighbourhood.
\end{enumerate}
\end{prp}

\begin{proof}\
\begin{itemize}
\item[\eqref{Ubasis}$\Rightarrow$\eqref{L(CH)}]  Immediate.

\item[\eqref{L(CH)}$\Rightarrow$\eqref{LCLH}]  Take $g\in G$ with a compact Hausdorff neighbourhood $C$ and an open neighbourhood $O$.  As $C$ is Hausdorff, for each $c\in C\setminus O$, we have disjoint open $O_c\ni g$ and $N_c\ni c$.  As $C\setminus O$ is closed in $C$ and $C$ is compact, $C\setminus O$ is also compact so we have some finite subcover $N_{c_1},\ldots,N_{c_k}$ of $C\setminus O$.  Let $O'=O_{c_1}\cap\ldots\cap O_{c_k}$.  As $C$ is compact, $\overline{O'}\cap C$ is compact.  Also $\overline{O'}\cap(N_{c_1}\cup\ldots\cup N_{c_k})=\emptyset$ so $g\in\overline{O'}\cap C\subseteq O$.  As $O$ was arbitrary, this shows that each $g$ has a neighbourhood base of compact subsets, as required.

\item[\eqref{LCLH}$\Rightarrow$\eqref{Ubasis}]  If $G$ is locally compact locally Hausdorff then the collection of all open subsets of $G$ contained in some compact Hausdorff $C$ forms a $\U$-basis.\qedhere
\end{itemize}
\end{proof}

Note \eqref{L(CH)}$\Rightarrow$\eqref{LCLH} is a minor extension of a standard argument for Hausdorff spaces.  Also note that \eqref{LCLH}$\Leftrightarrow$\eqref{L(CH)} is essentially saying
\[\text{locally compact locally Hausdorff}\quad\Leftrightarrow\quad\text{locally (compact Hausdorff)}.\]
It is also worth noting that in the non-Hausdorff settting, local compactness does not follow from compactness.  In fact, there is even an example in \cite{Scott2013} of a locally Hausdorff space that is compact but not locally compact.

\begin{thm}\label{Ubases}
Let $G$ be a topological space, and let $S$ be a $\cup$-basis $S$ of $G$. Then $S$ is a $\U$-basic poset (taking $\subseteq$ for $\leq$) such that, for any $O,N\in S$,
\begin{align}
\label{OprecN}O\prec N\qquad&\Leftrightarrow\qquad O\Subset N.\\
\label{OUN}O\U N\qquad&\Leftrightarrow\qquad O\cup N\text{ is Hausdorff}.
\end{align}
\end{thm}

\begin{proof}
First we note that $S$ is $\Supset$-round.  Indeed, by the equivalent version of \eqref{UnionClosedEquiv} above, any $O\in S$ must be compactly contained in some open Hausdorff $M$, i.e. $O\subseteq C\subseteq M$, for some compact $C$.  As $S$ is a basis, we can cover $C$ with elements of $S$ compactly contained in $M$.  As $C$ is compact, we can find a finite subcover whose union $F$ is compactly contained in $M$.  Thus $O\Subset F\in S$, by \eqref{UnionClosedEquiv}.

Now we show that $S$ is a $\U$-basic poset satisfying \eqref{OprecN} and \eqref{OUN}.

\begin{itemize}
\item[\eqref{OprecN}]  Take $O,N\in S$.  If $O\prec N$ then $O\Subset N$, by \eqref{OprecN=>}.  Conversely, if $O\Subset N$ then, for every $M\in S$ with $N\Subset M$, we have $\overline{O}\cap M\subseteq N$, by \eqref{OSubNequivs}.  Thus $\bigcup_{N\Subset M\in S}\overline{O}\cap M\subseteq N$ so $O\prec N$, by the other direction of \eqref{OprecN=>}.

\item[\eqref{succRound}]  As we already know $S$ is $\Supset$-round, this is immediate from \eqref{OprecN}.

\item[\eqref{precDistributivity}]  As $S$ is a basis and $G$ is locally compact, every $O\in S$ is the union of those elements of $S$ compactly contained within $O$.  Thus \eqref{Approximation} holds, which yields the $\Leftarrow$ part of \eqref{precDistributivity}, by \eqref{Approx=>}.

Conversely, take $M,N,O,O'\in S$ with $O'\Subset O\subseteq N\cup M\in S$, so we have compact $C$ with $O'\subseteq C\subseteq O$.  As $S$ is a basis, we can cover $C$ with elements of $S$ which are compactly contained within $O$ and either $N$ or $M$.  As $C$ is compact, we have a finite subcover.  As $S$ is a $\cup$-basis, we have $F\Subset N$, $G\Subset M$ and $F,G\in S$ for the unions of those elements in the finite subcover compactly contained within $N$ and $M$ respectively.  Moreover, $O'\Subset F\cup G\Subset N\cup M$ (and $F,G,F\cup G\in S$, as $N\cup M\in S$ and $S$ is a $\cup$-basis).  As $M$, $N$, $O$ and $O'$ were arbitrary, \eqref{precDistributivity} holds.

\item[\eqref{OUN}]  To prove \eqref{OUN}, take $O,N\in S$ and assume $O\U N$ but $O\cup N$ is not Hausdorff.  As $O$ and $N$ are Hausdorff, we must have $g\in O\setminus N$ and $h\in N\setminus O$ with no disjoint neighbourhoods.  Take $O',N'\in S$ with $g\in O'\Subset O$ and $h\in N'\Subset N$ so $O'\cap N'\Subset O,N$.  For every open $M\ni g$, $M\cap O'$ is a neighbourhood of $g$ so $M\cap O'\cap N'\neq\emptyset$, by the choice of $g$ and $h$.  Thus
\[g\in\overline{O'\cap N'}.\]

As $O\U N$ holds, we have $M\in S$ such that $P\Subset M\Subset O,N$, for all $P\in S$ with $P\Subset O',N'$.  As $O'\cap N'$ is the union of such $P$, we must have $O'\cap N'\subseteq M$ and hence $g\in\overline{M}$.  By \eqref{UInterpolation}, we also have $M'\in S$ with $M\Subset M'\Subset O,N$.  Thus $g\in\overline{M}\cap O\subseteq M'$, as $M\Subset M'\Subset O$ (see \eqref{OSubNequivs}), even though $g\notin N\supseteq M'$, a contradiction.

Conversely, take $O,N\in S$ and assume $O\cup N$ is Hausdorff.  Take $O',N'\in S$ with $O'\Subset O$ and $N'\Subset N$.  In particular, $O\cap N\Subset O\cup N$ so
\[C=\overline{O'\cap N'}\cap(O\cup N)\]
is compact (see \eqref{OSubNequivs}).  Moreover, as $O'\cap N'\Subset O,N$ and $O\cup N$ is a Hausdorff extension of both $O$ and $N$, we have $C\subseteq O\cap N$ (see the proof of \eqref{OSubNequivs}).  Thus we can cover $C$ with elements of $S$ compactly contained in $O\cap N$.  As $C$ is compact, we can take a finite subcover whose union $F$ is in $S$, as $S$ is a $\cup$-basis.  Then $O'\cap N'\Subset F\Subset O,N$ and hence, for all $M\in S$ with $M\Subset O',N'$, we have $M\Subset F$.  As $O'$ and $N'$ were arbitrary, this shows that $O\U N$ holds.

\item[\eqref{HausdorffJoins}]  Take $O,O',N,N'\in S$ with $O\Subset O'\U N'\Supset N$.  By what we just proved, $O'\cup N'$ is Hausdorff.  As $O\cup N\Subset O'\cup N'$ and $S$ is a $\cup$-basis, we must have $O\cup N\in S$, as required. \qedhere
\end{itemize}
\end{proof}

Incidentally, to get a $\cup$-basis $S$ of some compact locally Hausdorff $G$, it might be tempting to simply take the collection of all open Hausdorff subsets.  However, unless $G$ itself is Hausdorff, this may fail to yield a $\cup$-basis or a $\U$-basic poset and $\Subset$ may no longer correspond to $\prec$ on $S$ as above.

\begin{xpl}
Consider two copies of $[-1,1]$ and identify the strictly negative points to form a topological space $G=[-1,1]\sqcup[0,1]'$.  It is easy to see that $G$ is compact and locally Hausdorff but not Hausdorff, as the two copies $0$ and $0'$ of the origin can not be separated by open sets.  Let $S$ be the basis of $G$ consisting of all open Hausdorff subsets.  Then $O=G\setminus\{0'\}\in S$ even though the only compact subset containing $O$ is $G$, which is not Hausdorff.  Thus $S$ is not a $\cup$-basis.  Moreover, $O$ is a maximal element of $S$ and hence, trivially, $O\prec O$ in $S$, even though $O$ is not compact.  Also $O\prec O=[-1,1]\cup(0,1]'$ even though any $N\prec(0,1]'$ would have to avoid a neighbourhood of $0'$ (as $(0,1]'\subseteq[-1,0)\cup[0,1]'\in S$ so there needs to be some $N^\perp$ with $N^\perp\cap N=\emptyset$ and $[-1,0)\cup[0,1]'\subseteq N^\perp\cup(0,1]'$, which implies $0'\in N^\perp$).  In particular, we can not have both $N\prec(0,1]'$ and $O\subseteq[-1,1]\cup N$, i.e. $S$ fails to satisfy \eqref{precDistributivity} or even \eqref{Shrinking}.
\end{xpl}

\section{Filters}\label{secFilters}

\begin{dfn}
For any transitive relation $<$ on $S$, we call $U\subseteq S$ a \emph{$<$-filter} if
\begin{align}
\label{succClosed}\tag{$>$-Closed}v>u\in U\qquad&\Rightarrow\qquad v\in U\\
\label{precDirected}\tag{$<$-Directed}u,v\in U\qquad&\Rightarrow\qquad \exists w\in U(w<u,v).
\end{align}
\end{dfn}

Note we can extend any $U\subseteq S$ to a $>$-closed subset by taking the union with
\[U^<=\{a>u:u\in U\}.\]
Also denote the initial segment of $U$ defined by any $a\in S$ by
\[U^a=U\cap a^>=\{u\in U:u<a\}.\]
Initial segments will play an important role in \autoref{ssStoneSpaces}.  For the moment we just note that filters are determined by their initial segments.

\begin{prp}\label{InitialSegment}
For any $U\subseteq S$ and any transitive relation $<$ on $S$,
\[U\text{ is a $<$-filter}\qquad\Leftrightarrow\qquad\forall a\in U\ (U=U^{a<}).\]
\end{prp}

\begin{proof}
Say $U$ is a $<$-filter and $a\in U$.  Then $U$ is $<$-directed so, for every $u\in U$, we have $v\in U$ with $v<a,u$ which means $v\in U^a$ and hence $u\in U^{a<}$, i.e. $U\subseteq U^{a<}$.  On the other hand, $U^a\subseteq U$ and hence $U^{a<}\subseteq U$, as $U$ is $>$-closed, i.e. $U=U^{a<}$.

Conversely, assume $U$ satisfies the right side.  For any $a,b\in U$, we have $b\in U^{a<}$ so there must be some $u\in U^a$ with $u<b$.  But this means $u\in U$ and $u<a$, so $U$ is $<$-directed.  As $U=U^{a<}$, $U$ is also $>$-closed and hence a $<$-filter.
\end{proof}

\subsection{Ultrafilters}\label{subsecUltrafilters}

\begin{dfn}
A \emph{$<$-ultrafilter} is a maximal proper $<$-filter.
\end{dfn}

Ultrafilters can be used to recover the points of a space from a basis, as shown by the following generalization of \cite[Proposition 3.2]{BiceStarling2016}.

\begin{prp}\label{SubsetUltrafilters}
If $S\ni\emptyset$ is a Hausdorff basis of some locally compact $G$ then
\[g\mapsto U_g=\{O\in S:g\in O\}\]
is a bijection from $G$ to $\Subset$-ultrafilters of $S$.
\end{prp}

\begin{proof}
If $U$ is a $\Subset$-ultrafilter of $S$ then take $M\in U$.  By \autoref{InitialSegment} above, for $U^M=U\cap M^\Supset=\{O\in U:O\Subset M\}$, we have
\[\bigcap U=\bigcap U^M\subseteq\bigcap_{O\in U^M}\overline{O}\cap M=\bigcap_{\substack{O\in U^M\\ O\Subset N\in U^O}}\overline{O}\cap M\subseteq\bigcap_{N\in U^M}N=\bigcap U^M,\]
where the last $\subseteq$ follows form \eqref{OSubNequivs}.  As an intersection of non-empty compact closed subsets of $O$ with the finite intersection property, $\bigcap U$ must be non-empty.  For any $g\in\bigcap U$, we have $U\subseteq U_g$ and hence, by maximality, $U=U_g$.

Conversely, as $G$ is locally compact, each $U_g$ is a $\Subset$-filter and hence has some maximal extension $U'\supseteq U_g$ not containing $\emptyset$, by the Kuratowski-Zorn lemma.  By what we just proved, $U'=U_{g'}$ for some $g'\in G$.  But as $G$ is locally Hausdorff and hence $T_1$, $U_g\subseteq U_{g'}$ implies $g=g'$.  Thus $U_g$ was already a $\Subset$-ultrafilter.
\end{proof}

Now take a poset $S$.  As the rather below relation $\prec$ is auxiliary to $\leq$,
\[U\text{ is a $\prec$-filter}\qquad\Leftrightarrow\qquad U\text{ is a $\prec$-round $\leq$-filter}.\]
As $0\prec0$, the proper $\prec$-filters are precisely those that do not contain $0$.  Thus every $\prec$-filter not containing $0$ is contained in a $\prec$-ultrafilter, again by the Kuratowski-Zorn lemma.  In fact, the same applies to $\prec$-directed subsets, as the upwards $\prec$-closure of a $\prec$-directed subset is a $\prec$-filter.

It is well-known that proper filters in Boolean algebras are maximal (i.e. ultrafilters) iff they are prime iff they intersect every complementary pair.  This generalizes to basic posets and even $\prec$-distributive conditional $\vee$-semilattices as follows.

\begin{thm}\label{ultrachars}
If $U$ is a proper $\prec$-filter in a $\prec$-distributive conditional $\vee$-semilattice $S$ then $U$ being maximal is equivalent to each of the following conditions.
\begin{align}
\label{Complementary}\tag{Complementary}a\prec b\notin U\qquad&\Rightarrow\qquad\exists c\in U\ (a\perp c)\quad\text{or}\quad\forall c\in U\ (b\not\smallsmile c).\\
\label{Prime}\tag{Prime}a\vee b\in U\qquad&\Rightarrow\qquad a\in U\text{ or }b\in U.
\end{align}
\end{thm}

\begin{proof}
Assume that \eqref{Complementary} holds but $U$ is not maximal, so we have a proper $\prec$-filter $V$ properly containing $U$.  Take $v\in V\setminus U$ and $u\in U$.  Then we have $w\in V$ with $w\prec v,u$.  In particular, $w\prec v\notin U$.  Having $c\in U\subseteq V$ with $w\perp c$ would contradict the fact $V$ is $\prec$-filter, so \eqref{Complementary} implies that $w\not\smallsmile c$, for all $c\in U$.  In particular, $w\not\smallsmile u$, even though $w\leq u$, contradicting \eqref{leqU}.  Thus \eqref{Complementary} implies maximality.

Conversely, assume that $U$ is a $\prec$-filter failing \eqref{Complementary}, so we have $a,b,c\in S$ with $a\prec b\notin U$, $b\smallsmile c\in U$ and $a\not\perp u$, for all $u\in U$.  To see that $U$ is not maximal, let
\[V=\{v\in S:\exists d\ (a\prec d\prec b\text{ and }\exists u\in U\ \forall e\prec d,u\ (e\prec v))\}.\]
We immediately see that $V$ is $\succ$-closed.  Also $0\notin V$, otherwise we would have $d$ with $a\prec d\prec b$ and $u\in U$ such that $e\prec d,u$ is only possible for $e=0$.  By \eqref{perpEquivalent}, this would mean $d\perp u$ and hence $a\perp u$, a contradiction.  Next we claim that $V$ is $\prec$-directed.  To see this, take $v,v'\in V$ and corresponding $d,d'\in S$ with $a\prec d,d'\prec b$ and $u,u'\in U$ with $e\prec v$, for all $e\prec d,u$, and $e\prec v'$, for all $e\prec d',u'$.  As $d,d'\prec b$, \eqref{LocallyHausdorff} from \autoref{LocallyHausdorffprp} yields $d\U d'$ and then \eqref{UInterpolation} yields $d'',d'''\in S$ with $a\prec d'''\prec d''\prec d,d'$.  As $U$ is a $\prec$-filter, we also have $u''',u''\in U$ with $u'''\prec u''\prec u,u',c$.  As $d''\prec b\U c\succ u''$, \eqref{UAuxiliarity} implies $d'''\prec d''\U u''\succ u'''$, so we have $v''\prec d'',u''$ with $e\prec v''$, for all $e\prec d''',u'''$.  Thus $v''\prec d''\prec d',d$ and $v''\prec u''\prec u',u$ so $v''\prec v',v$.  Also $d'''$ and $u'''$ witness $v''\in V$ so the claim is proved and we have shown that $V$ is a proper $\prec$-filter.  Moreover, $U\subseteq V$ and $b\in V\setminus U$, so $V$ is a proper extension of $U$, i.e. $U$ is not maximal.  In other words, maximality implies \eqref{Complementary}.

\begin{itemize}
\item[\eqref{Prime}$\Rightarrow$\eqref{Complementary}]  Assume \eqref{Prime}, $a\prec b\notin U$ and $b\U u\in U$.  We have to find $v\in U$ with $a\perp v$.  Take $u'\in U$ with $u'\prec u$.  As $a\prec b\U u\succ u'$, we have $c'\prec b,u$ with $e\prec c'$, for all $e\prec a,u'$.  By \eqref{UInterpolation}, we have $c$ with $c'\prec c\prec b,u$.  As $c'\prec c\leq u$, \eqref{RatherBelow} yields $d\perp c'$ such that $c\vee d\geq u\in U$.  As $c\leq b\notin U$, we must have $d\in U$, by \eqref{Prime}.  As $U$ is a $\prec$-filter, we have $v\in U$ with $v\prec d,u'$.  Then, for any $e\prec a,v$, we have $e\prec a,u'$ and hence $e\prec c'\perp d\succ v\succ e$ so $e=0$.  By \eqref{perpEquivalent}, $a\perp v\in U$ as required.

\item[\eqref{Complementary}$\Rightarrow$\eqref{Prime}]  Assume \eqref{Complementary} and $a\vee b\in U$ but $a,b\notin U$, looking for a contradiction.  Take $c\in U$ with $c\prec a\vee b$.  By \eqref{Shrinking}, we have $a'\prec a$ and $b'\prec b$ with $c\leq a'\vee b'$ and hence $a'\vee b'\in U$.  By \eqref{leqU}, $a'\U a\vee b\in U$ so, as $a'\prec a\notin U$, \eqref{Complementary} yields $v\in U$ with $a'\perp v$.  Likewise, we have $w\in U$ with $b'\perp w$.  As $U$ is a $\prec$-filter, we have $x\in U$ with $x\prec v,w$ and hence $x\perp a'\vee b'$, by \eqref{veePreservation}, contradicting $a'\vee b'\in U$.\qedhere
\end{itemize}
\end{proof}

\subsection{Ultrafilter Spaces}\label{subsecStoneSpaces}

\begin{center}
\textbf{Let $G$ be the set of $\prec$-ultrafilters of a basic poset $S$.}
\end{center}

We consider $G$ with the topology generated by $(O_a)_{a\in S}$ where
\[O_a=\{U\in G:a\in U\}.\]
Note this topology makes the map $g\mapsto U_g$ in \autoref{SubsetUltrafilters} a homeomorphism (as the map then takes sets in the given basis to the subbasis $(O_a)_{a\in S}$).

\begin{dfn}
We call $G$ the \emph{$\prec$-ultrafilter space} of $S$.
\end{dfn}

Note that when $S$ is a Boolean algebra, $G$ is just the Stone space of $S$.  Our goal is to show that $\prec$-ultrafilter space $G$ is still locally compact and locally Hausdorff for general basic poset $S$ (although, in contrast to Stone spaces, $G$ may not be $0$-dimensional).  First we note some basic facts about $G$.

\begin{prp}\label{basicGfacts}
The sets $(O_a)_{a\in S}$ form a basis for $G$ and, for all $a,b\in S$,
\begin{align}
\label{veecup}O_{a\vee b}&=O_a\cup O_b.\\
\label{perpperp}a\perp b\qquad&\Leftrightarrow\qquad O_a\cap O_b=\emptyset.
\end{align}
\end{prp}

\begin{proof}
If $U\in O_a\cap O_b$ then $a,b\in U$ so, as $U$ is a $\prec$-filter, we have $c\in U$ with $c\prec a,b$ and hence $U\in O_c\subseteq O_a\cap O_b$.  Thus $(O_a)_{a\in S}$ form a basis.
\begin{itemize}
\item[\eqref{veecup}]  Immediate from \eqref{Prime}.
\item[\eqref{perpperp}]  If $a\perp b$ then certainly $O_a\cap O_b=\emptyset$.  Conversely, if $a\not\perp b$ then \eqref{perpEquivalent} and \autoref{Snot0} yields a $\prec$-decreasing sequence $a,b\succ c_1\succ c_2\succ\ldots$ of non-zero elements, which extends to a $\prec$-ultrafilter in $O_a\cap O_b$.\qedhere
\end{itemize}
\end{proof}

Before proving our main theorem we need the following result.

\begin{lem}\label{abC}
Assume $a\prec b$ and $C\subseteq S$.  If $a\not\prec\bigvee F$, for any finite $F\subseteq C$, then
\[\overline{O_a}\cap O_b\not\subseteq\bigcup_{c\in C}O_c.\]
\end{lem}

\begin{proof}
First note that we can replace $C$ with $C'=C^\succ\cap b^\succ=\{d\prec b,c:c\in C\}$.  Indeed, if we had finite $F'\subseteq C'$ with $a\prec\bigvee F'$ then we would have finite $F\subseteq C$ with $F'\subseteq F^\succ$ and hence $a\prec\bigvee F$, a contradiction.  Moreover,
\[O_b\cap\bigcup_{c\in C}O_c=O_b\cap\bigcup_{c'\in C'}O_{c'},\]
as any $\prec$-filter containing $b$ and $c\in C$ must also contain some $c'\prec b,c$, and hence
\[\overline{O_a}\cap O_b\setminus\bigcup_{c\in C}O_c=\overline{O_a}\cap O_b\setminus\bigcup_{c'\in C'}O_{c'}.\]
In particular, if the right side is non-empty then so is the left side.

Thus we may assume $C$ is actually bounded by $b$, i.e. $C\subseteq b^\succ$.  Now consider
\[D=\{d\leq b:F\subseteq C\text{ is finite and }a\prec d\vee\bigvee F\}.\]
First we see that $0\notin D$, as $a\not\prec\bigvee F$, for any finite $F\subseteq C$.  Next we claim that $D$ is $\prec$-directed.  To see this, take $x,y\in D$, so we have finite $F,G\subseteq C$ with $a\prec x\vee\bigvee F,y\vee\bigvee G\leq b$.  By \eqref{LocallyHausdorff} from \autoref{LocallyHausdorffprp}, $x\vee\bigvee F\U y\vee\bigvee G$ so \eqref{UInterpolation} yields $z$ with $a\prec z\prec x\vee\bigvee F,y\vee\bigvee G$.  By \autoref{precabcd}, we have $z'\prec x,y$ with $a\prec z'\vee\bigvee(F\cup G)\leq b$.  This proves the claim, so $D$ can be extended to a $\prec$-ultrafilter $U$.  As $b\in D\subseteq U$ (e.g. taking $F=\emptyset$), we have $U\in O_b$.

Looking for a contradiction, assume that $U\notin\overline{O_a}$.  Thus we have $u\in U$ with $O_a\cap O_u=\emptyset$ and hence $a\perp u$, by \eqref{perpperp}.  As $b\in U$, by lowering $u$ if necessary, we may assume that $u\leq b$.  Take $w\in U$ with $w\prec u$.  By \eqref{RatherBelow}, we have $x\perp w$ with $b\leq x\vee u$.  As $a\prec b$, \eqref{precDistributivity} yields $x'\prec x$ and $u'\prec u$ with $a\prec x'\vee u'\prec b$.  As $a\perp u\succ u'$, \eqref{perpDistributivity} implies $a\prec x'$ and hence $x'\in D\subseteq U$ (witnessed by $F=\emptyset$).  But this contradicts $x'\prec x\perp w\in U$.

Lastly, take any $x\prec c\in C$.  By \eqref{RatherBelow}, we have $d\perp x$ with $a\prec b\leq c\vee d$ so \eqref{precDistributivity} yields $c'\prec c$ and $d'\prec d$ with $a\prec c'\vee d'\prec b$.  Thus $a\prec c\vee d'\leq b$ so $d'\in D\subseteq U$ and hence $d\in U$, which in turn yields $x\notin U$.  As $x$ was arbitrary, $c\notin U$, i.e. $U\notin O_c$.  Putting this all together, we see that
\[U\in\overline{O_a}\cap O_b\setminus\bigcup_{c\in C}O_c.\qedhere\]
\end{proof}

\begin{thm}\label{SpacesFromPosets}
Let $S$ be a basic poset and let $G$ be the set of $\prec$-ultrafilters of $S$. Then $G$ is a locally compact locally Hausdorff space and
\begin{align}
\label{inclusion}a\leq b\qquad  &\Leftrightarrow\qquad O_a\subseteq O_b.\\
\label{Hunion}a\smallsmile b\qquad  &\Leftrightarrow\qquad O_a\cup O_b\text{ is Hausdorff}.\\
\label{compcont}a\prec b\qquad  &\Leftrightarrow\qquad O_a\Subset O_b.
\end{align}
Moreover, if $S$ is a $\U$-basic poset then $(O_a)_{a\in S}$ is a $\cup$-basis for $G$.
\end{thm}

\begin{proof}\
\begin{itemize}
\item[\eqref{inclusion}]  If $a\leq b$ then certainly $O_a\subseteq O_b$.  Conversely, if $a\nleq b$ then we have $a'\prec a$ with $a'\not\prec b$, by \eqref{LowerOrder}.  By \autoref{abC}, with $a'$, $a$ and $b$ taking the place $a$, $b$ and $C$ respectively, $\overline{O_{a'}}\cap O_a\nsubseteq O_b$ and hence $O_a\nsubseteq O_b$.

\item[$\Rightarrow$ in \eqref{Hunion}]  Assume $a\smallsmile b$ and take any distinct $U\in O_a$ and $V\in O_b$.  By maximality, $V\nsubseteq U$ so we can take $v\in V\setminus U$.  Further take $v',v''\in V$ with $v''\prec v'\prec v,b$ so $v''\prec v'\notin U$ and $v'\prec b\U a$.  By \eqref{UAuxiliarity}, $v'\U a\in U$ so \eqref{Complementary} yields $c\in U$ with $v''\perp c$.  Thus $O_c$ and $O_{v''}$ are disjoint neighbourhoods of $U$ and $V$ respectively.  As $a\U a$ and $b\U b$, the same argument would apply to distinct filters both in $O_a$ or $O_b$.

\item[\eqref{compcont}]  Take $a\prec b$.  First note that, for any $c\leq b$ or even $c\U b$,
\begin{equation}\label{ullc}
a\prec c\qquad\Rightarrow\qquad\overline{O_a}\cap O_b\subseteq O_c.
\end{equation}
Indeed, $U\in\overline{O_a}$ means that, for every $d\in U$, $O_d\cap O_a\neq\emptyset$ so $d\not\perp a$, by \eqref{perpperp}.  If we also had $U\in O_b\setminus O_c$, i.e. $a\prec c\notin U$ but $c\smallsmile b\in U$, then \eqref{Complementary} would yield $d\in U$ with $d\perp a$, a contradiction.  

Now we want to show that $\overline{O_a}\cap O_b$ is compact and hence $O_a\Subset O_b$.  For this, we need to show that if $(O_c)_{c\in C}$ has no finite subcover of $\overline{O_a}\cap O_b$ then the entirety of $(O_c)_{c\in C}$ does not cover $\overline{O_a}\cap O_b$ either.  In fact it suffices to consider $C\subseteq b^\geq$, as $(O_c)_{c\leq b}$ forms a basis for the subspace topology on $O_b$ (as each $U\in O_b$ is $\prec$-directed).  But then having no finite subcover implies that $a\not\prec\bigvee F$, for any finite $F\subseteq C$.  Indeed, if we had $a\prec\bigvee F$, for some finite $F\subseteq C\subseteq b^\geq$, then \eqref{veecup} and \eqref{ullc} would yield
\[\overline{O_a}\cap O_b\subseteq O_{\bigvee F}=\bigcup_{f\in F}O_f,\]
a contradiction.  Thus the entirety of $(O_c)_{c\in C}$ does not cover $\overline{O_a}\cap O_b$ either, by \autoref{abC}, showing that $\overline{O_a}\cap O_b$ is indeed compact.

Conversely, say $a\not\prec b$.  If $a\nleq b$ then \eqref{inclusion} yields $O_a\nsubseteq O_b$ and so certainly $O_a\not\Subset O_b$.  If $a\leq b$ then, by \eqref{succRound}, we can take $c\succ b$, which then implies $a\prec c$ and hence $O_a\Subset O_c$.  By \autoref{abC}, with $c$ and $b$ taking the place of $b$ and $C$ respectively, $\overline{O_a}\cap O_c\nsubseteq O_b$.  As $c\U c$, the $\Rightarrow$ part of \eqref{Hunion} implies that $O_c$ is Hausdorff so compact subsets are closed.  Thus $\overline{O_a}\cap O_c$ is the smallest compact subset of $O_c$ containing $O_a$ so $O_a\not\Subset O_b$.

\item[$\Leftarrow$ in \eqref{Hunion}]  Now we know that $\prec$ corresponds to $\Subset$, we can argue as in the proof of \eqref{OUN} from \autoref{Ubases}.
\end{itemize}

By \eqref{compcont} $G$ is locally compact.  Indeed, if $U\in O_a$ then we have $b\in U$ with $b\prec a$ so $U\in O_b\Subset O_a$.  Likewise, by \eqref{Hunion} and the reflexivity of $\U$, $G$ is locally Hausdorff.

Now assume $S$ is a $\U$-basic poset.  To see that $(O_a)_{a\in S}$ is a $\cup$-basis, we need to verify \eqref{UnionClosedEquiv}.  For the $\Rightarrow$ part, say $O_a\cup O_b\subseteq O_c$, for some $a,b,c\in S$.  By \eqref{succRound}, we have $d\succ c$ and hence $O_a\cup O_b\subseteq O_c\Subset O_d$, as required.

Conversely, to verify the $\Leftarrow$ part of \eqref{UnionClosedEquiv}, take $a,b\in S$ such that $O_a\cup O_b\Subset M$, for some open Hausdorff $M$.  So we have compact $C$ with $O_a\cup O_b\subseteq C\subseteq M$.  For each $U\in C$ we have $c\in U$ with $O_c\subseteq M$.  Then we can take $c'\in U$ with $c'\prec c$.  As $C$ is compact, we can cover $C$ with finitely many such subsets $O_{c'}$.

Say $C$ is covered by just pair of subsets $O_{c'}$ and $O_{d'}$.  Then $c'\prec c\U d\succ d'$, as $O_c\cup O_d\subseteq M$ is Hausdorff, and hence $c'\vee d'$ exists in $S$, by \eqref{HausdorffJoins}.  As $O_a,O_b\subseteq C\subseteq O_{c'\vee d'}$, we have $a,b\leq c'\vee d'$ and hence $a\vee b$ exists too, as $S$ is a conditional $\vee$-semilattice.

Now say $C$ is covered by $3$ such subsets, $O_{c_1'}$, $O_{c_2'}$ and $O_{c_3'}$.  By \eqref{Interpolation}, we have $x_1,x_2\in S$ with $c_1'\prec x_1\prec c_1$ and $c_2'\prec x_2\prec c_2$.  Again by \eqref{HausdorffJoins}, $c_1'\vee c_2'$ and $x_1\vee x_2$ exist and, by \autoref{aveebllc}, $c_1'\vee c_2'\prec x_1\vee x_2$.  As $O_{x_1\vee x_2}\cup O_{c_3}$ is Hausdorff, still being contained in $M$, we can apply \eqref{HausdorffJoins} again to show that $c_1'\vee c_2'\vee c_3'$ exists in $S$.  Again $O_a,O_b\subseteq C\subseteq O_{c_1'\vee c_2'\vee c_3'}$ shows that $a,b\leq c_1'\vee c_2'\vee c_3'$ so $a\vee b$ exists, as $S$ is a conditional $\vee$-semilattice.  Extending by induction shows that, regardless of how many subsets are needed to cover $C$, $O_a\cup O_b=O_{a\vee b}$ is always in the basis $(O_a)_{a\in S}$, as required.
\end{proof}

%We now use this representation to strengthen \autoref{llexistsUSS} by removing the $\U=S\times S$ assumption.  It would be nice to know if a more direct proof is possible.
%
%\begin{cor}\label{llexists}
%If $S$ is a basic poset then
%\[a\prec b\qquad\Leftrightarrow\qquad\exists c\succ a,b\ \exists a'\perp a\ (c\leq a'\vee b).\]
%\end{cor}
%
%\begin{proof}
%If $a,b\prec c$, $a'\perp a$ and $c\leq a'\vee b$ then, by \autoref{SpacesFromPosets},
%\[O_a,O_b\Subset O_c,\qquad O_{a'}\cap O_a=\emptyset\qquad\text{and}\qquad O_c\subseteq O_{a'}\cup O_b.\]
%As $O_a\Subset O_c$, it follows that $\overline{O_a}\cap O_c$ is compact.  Also $O_{a'}\cap\overline{O_a}=\emptyset$ so
%\[O_a\subseteq\overline{O_a}\cap O_c\subseteq\overline{O_a}\cap(O_{a'}\cup O_b)=\overline{O_a}\cap O_b\subseteq O_b.\]
%Thus $O_a\Subset O_b$ and hence $a\prec b$, again by \autoref{SpacesFromPosets}.
%\end{proof}

\section{Functoriality}\label{secFunctoriality}

Up till now we have focused on certain structures rather than the maps between them.  However, the classic Stone duality is also functorial with respect to the appropriate morphisms.  Specifically, any continuous map $F:G\rightarrow H$, for zero dimensional compact Hausdorff spaces $G$ and $H$, yields a Boolean homomorphism $\pi:\mathrm{Clopen}(H)\rightarrow\mathrm{Clopen}(G)$ defined by
\[\pi(O)=F^{-1}[O].\]

However, this does not extend to general bases $S$ and $T$ of $G$ and $H$ for the simple reason that $F^{-1}$ may not take elements of $T$ to elements of $S$.  To avoid this problem, we instead consider the relation $\sqsubset$ between $S$ and $T$ defined by
\[O\sqsubset N\qquad\Leftrightarrow\qquad O\subseteq F^{-1}[N].\]
Fortunately, such relations admit a first order characterisation.

\subsection{Basic Morphisms}\label{subsecBasicMorphisms}

%Throughout this section,
%\begin{center}
%\textbf{assume $S$ and $S'$ are basic posets and $\sqsubset\ \subseteq S\times S'$}.
%\end{center}

\begin{dfn}\label{BasicMorph}
For posets $S$ and $S'$, we call $\sqsubset\ \subseteq S\times S'$ a \emph{basic morphism} if
\begin{align}
\label{Faithful}\tag{Faithful}a\sqsubset 0&\qquad\Rightarrow\qquad a=0.\\
\label{2Auxiliarity}\tag{Auxiliarity}a\leq b\sqsubset b'\leq a'&\qquad\Rightarrow\qquad a\sqsubset a'.\\
\label{Pushforward}\tag{Pushforward}a\prec b\sqsubset c',d'&\qquad\Rightarrow\qquad\exists b'\ (a\sqsubset b'\prec c',d').\\
\label{veePullback}\tag{$\vee$-Pullback}a\prec b\sqsubset c'\vee d'&\qquad\Rightarrow\qquad\exists c\sqsubset c'\ \exists d\sqsubset d'\ (a\prec c\vee d).\\
\intertext{We call $\sqsubset$ a \emph{basic $\vee$-morphism} if, moreover, whenever the given joins exist,}
\label{veePreserving}\tag{$\vee$-Preserving}a\sqsubset a'\quad\text{and}\quad b\sqsubset b'&\qquad\Rightarrow\qquad a\vee b\sqsubset a'\vee b'.\\
\label{LowerRelation}\tag{Lower Relation}\forall b\prec a\ (b\sqsubset a')&\qquad\Rightarrow\qquad a\sqsubset a'.\\
\label{Bottom}\tag{Bottom}a'\in S'&\qquad\Rightarrow\qquad 0\sqsubset a'.
\end{align}
\end{dfn}

As in the previous section, for any $U\subseteq S$, let $U^\sqsubset=\{a'\sqsupset a:a\in U\}$.

\begin{thm}\label{ctsphi}
Let $G$ and $G'$ be topological spaces, and suppose $S$ and $S'$ are $\cup$-bases of $G$ and $G'$ respectively.  Then any continuous $\phi:D\rightarrow G'$, for $D\subseteq G$, defines a basic $\vee$-morphism $\sqsubset\ \subseteq S\times S'$ by
\begin{equation}\label{sqsubsetdef}
a\sqsubset a'\qquad\Leftrightarrow\qquad a\subseteq\phi^{-1}[a'].
\end{equation}
Moreover, for every $\prec$-ultrafilter $U\subseteq S$,
\begin{align}
\label{DInterior}\bigcap U\in D^\circ\qquad&\Leftrightarrow\qquad U^\sqsubset\neq\emptyset,\\
\label{FRecovery}\text{in which case}\qquad \phi(\bigcap U)&=\bigcap U^\sqsubset.
\end{align}
\end{thm}

\begin{proof} First we verify the required properties of $\sqsubset$.
\begin{itemize}
\item[\eqref{Faithful}]  If $a\subseteq\phi^{-1}[\emptyset]=\emptyset$ then $a=\emptyset$.

\item[\eqref{2Auxiliarity}]  If $a\leq b\sqsubset b'\leq a'$ then $a\subseteq b\subseteq\phi^{-1}[b']\subseteq\phi^{-1}[a']$ and hence $a\subseteq\phi^{-1}[a']$.

\item[\eqref{Pushforward}]  If $a\Subset b\subseteq \phi^{-1}[c']\cap \phi^{-1}[d']$ then we have compact $C\subseteq G$ with $a\subseteq C\subseteq b$.  As $\phi$ is continuous, $\phi[C]$ is also compact and $\phi[C]\subseteq\phi[b]\subseteq c'\cap d'$.  Thus we can cover $\phi[C]$ by a finite collection of elements in the basis $S'$ which are compactly contained in $c'\cap d'$.  Taking their union, we obtain $b'\in S'$ with $\phi[a]\subseteq\phi[C]\subseteq b'\Subset c'\cap d'$ and hence $a\sqsubset b'\prec c',d'$.

\item[\eqref{veePullback}]  If $a\Subset b\subseteq \phi^{-1}[c'\cup d']=\phi^{-1}[c']\cup\phi^{-1}[d']$ then we can argue as in the proof of \eqref{precDistributivity} in \autoref{Ubases} to obtain $c\Subset \phi^{-1}[c']$ and $d\Subset \phi^{-1}[d']$ within $b$ with $a\Subset c\cup d$ and hence $a\prec c\vee d$.

\item[\eqref{veePreservation}]  If $a\subseteq \phi^{-1}[a']$ and $b\subseteq \phi^{-1}[b']$ then $a\cup b\subseteq \phi^{-1}[a']\cup \phi^{-1}[b']=\phi^{-1}[a'\cup b']$.

\item[\eqref{LowerRelation}]  If $b\subseteq \phi^{-1}[a']$, for all $b\Subset a$, then $a=\bigcup_{b\Subset a}b\subseteq \phi^{-1}[a']$.

\item[\eqref{DInterior} and \eqref{FRecovery}]  If $\bigcap U\notin D^\circ$ then no $u\in U$ is contained in $D$ and so we can not have $u\subseteq \phi^{-1}[a'](\subseteq D)$ for any $a'\in S'$, i.e. $U^\sqsubset=\emptyset$.  While if $g\in\bigcap U\in D^\circ$ then $\phi^{-1}[a']$ is a neighbourhood of $g$ whenever $\phi(g)\in a'\in S'$.  By \autoref{SubsetUltrafilters}, $U=U_g=$ all those neighbourhoods of $g$ in $S$ so we must have some $u\in U$ with $g\in u\subseteq \phi^{-1}[a']$.  Thus $U^\sqsubset\neq\emptyset$ and $\phi(\bigcap U)=\phi(g)=\bigcap U^\sqsubset$, as $a'$ was arbitrary and $G'$ is locally Hausdorff and hence $T_1$.\qedhere

\item[\eqref{Bottom}]  Just note $\emptyset\subseteq\phi^{-1}[a']$, for any $a'\in S'$.
\end{itemize}
\end{proof}

Any relation $\sqsubset\ \subseteq S\times S'$ between posets $S$ and $S'$ has an extension $\sqsubseteq$ defined by
\[a\sqsubseteq a'\qquad\Leftrightarrow\qquad\forall b\prec a\ \exists\text{ finite }F\sqsubset a'\ (b\prec\bigvee F),\]
where $F\sqsubset a'$ means $a\sqsubset a'$, for all $a\in F$, i.e. $F\subseteq a'^\sqsupset$.

\begin{prp}
If $\sqsubset$ is a basic $\vee$-morphism between basic posets $S$ and $S'$ then
\[\sqsubseteq\ =\ \sqsubset.\]
\end{prp}

\begin{proof}
Taking $a$ for $F$ shows that $\sqsubset\ \subseteq\ \sqsubseteq$.  Now assume $\sqsubset$ is a basic $\vee$-morphism and $a\sqsubseteq a'$ so, for every $b\prec a$, we have finite $F\sqsubset a'$ with $b\prec\bigvee F$.  If $F=\emptyset$ then $b\prec\bigvee\emptyset=0$ so $b=0\sqsubset a'$, by \eqref{Bottom}.  If $F\neq\emptyset$ then $\bigvee F\sqsubset a'$, by \eqref{veePreserving}, so again $b\sqsubset a'$, by \eqref{2Auxiliarity}.  Thus $a\sqsubset a'$, by \eqref{LowerRelation}.
\end{proof}

In fact, even when $\sqsubset$ is only a basic morphism between basic posets, $\sqsubseteq$ will always be a basic $\vee$-morphism extension.  To prove this, we could work directly with the posets $S$ and $S'$.  Alternatively, one can go through the $\prec$-ultrafilter spaces and note that this follows from \autoref{ctsphi} above and \eqref{Sqsubseteq} below.

\begin{thm}\label{Stonects}
Let $S$ and $S'$ be basic posets, and let $G$ and $G'$ be their $\prec$-ultrafilter spaces.  Then any basic morphism $\sqsubset$ defines a continuous map $\phi$ from the open subset $D=\{U\in G:U^\sqsubset\neq\emptyset\}$ to $G'$ by
\begin{equation}\label{phidef}
\phi(U)=U^\sqsubset.
\end{equation}
We can characterize when $\phi$ is defined everywhere (i.e. $D=G$) as follows.
\begin{equation}\label{Cofinality}
\forall U\in G\ (U^\sqsubset\neq\emptyset)\qquad\Leftrightarrow\qquad\forall a\in S\ \exists\text{ finite }F\subseteq S'^\sqsupset\ (a\prec\bigvee F).
\end{equation}
Moreover, for any $a\in S$ and $a'\in S'$,
\begin{equation}\label{Sqsubseteq}
a\sqsubseteq a'\qquad\Leftrightarrow\qquad O_a\subseteq\phi^{-1}[O_{a'}].
\end{equation}
\end{thm}

\begin{proof}
First note that $U^\sqsubset\neq\emptyset$ means $a\sqsubset a'$, for some $a\in U$ and $a'\in S'$.  Thus $U^\sqsubset\neq\emptyset$ for any $U\in O_a$, so $D$ is indeed open.

Next we show that $U^\sqsubset\in G'$, for any $U\in D$.  By \eqref{Faithful}, $0\notin U^\sqsubset$, as $0\notin U$.  If $a',b'\in U^\sqsubset$ then we have $a,b\in U$ with $a\sqsubset a'$ and $b\sqsubset b'$.  Then we have $c,d\in U$ with $c\prec d\prec a,b$.  By \eqref{2Auxiliarity}, $c\prec d\sqsubset a',b'$ so we have $d'$ with $c\sqsubset d'\prec a',b'$, by \eqref{Pushforward}, and hence $d'\in U^\sqsubset$.  Thus $U^\sqsubset$ is $\prec$-directed and also $\succ$-closed, again by \eqref{2Auxiliarity}.  Lastly, if $a'\vee b'\in U^\sqsubset$ then we have $c,d\in U$ with $c\prec d\sqsubset a'\vee b'$.  By \eqref{veePullback}, we have $a\sqsubset a'$ and $b\sqsubset b'$ with $c\prec a\vee b$.  By \eqref{Prime}, either $a\in U$ or $b\in U$ and hence $a'\in U^\sqsubset$ or $b'\in U^\sqsubset$.  Thus $U^\sqsubset$ also satisfies \eqref{Prime} and hence $U^\sqsubset$ is a $\prec$-ultrafilter, i.e $U^\sqsubset\in G'$.

To see that $\phi$ is continuous note that, for any $a'\in S'$,
\[\phi^{-1}[O_{a'}]=\bigcup_{a\sqsubset a'}O_a.\]
In particular, $\phi^{-1}[O_{a'}]$ is always open.

\begin{itemize}
\item[\eqref{Cofinality}]  If the right side of \eqref{Cofinality} holds then, whenever $a\in U\in G$, we can take finite $F\subseteq S'^\sqsupset$ with $a\prec\bigvee F$.  By \eqref{Prime}, we have some $b\in F\cap U$ and then we have some $a'\in S'$ with $b\sqsubset a'$.  Thus $a'\in U^\sqsubset\neq\emptyset$.  Conversely, if the right side of \eqref{Cofinality} fails then we have some $a\in S$ such that $a\not\prec\bigvee F$, for all finite $F\subseteq S'^\sqsupset$.  By \eqref{succRound}, we have some $b\succ a$.  By \autoref{abC} with $C=S'^\sqsupset$, we have some
\[U\in\overline{O_a}\cap O_b\setminus\bigcup_{c\in C}O_c\subseteq G\setminus\bigcup_{c\sqsubset a'\in S'}O_c.\]
Thus $U^\sqsubset=\emptyset$.

\item[\eqref{Sqsubseteq}]  If $a\sqsubseteq a'$ and $U\in O_a$ then we have $b\in U$ with $b\prec a$.  Thus we have finite $F\sqsubset a'$ with $b\prec\bigvee F$ and hence we have some $c\in F\cap U$.  This means $a'\in c^\sqsubset\subseteq U^\sqsubset=\phi(U)$, i.e. $\phi(U)\in O_{a'}$ so $U\in\phi^{-1}[O_{a'}]$.  Conversely, if $a\not\sqsubseteq a'$ then we have some $b\prec a$ such that $b\not\prec\bigvee F$, for any finite $F\sqsubset a'$.  By \autoref{abC} with $a$ and $b$ switched and $C=a'^\sqsupset$,
\[O_a\supseteq\overline{O_b}\cap O_a\nsubseteq\bigcup_{c\in C}O_c=\bigcup_{c\sqsubset a'}O_c=\phi^{-1}[O_{a'}].\]
Thus $O_a\nsubseteq\phi^{-1}[O_{a'}]$.\qedhere
\end{itemize}
\end{proof}

\subsection{Composition}\label{subsecComposition}

%Now assume
%\begin{center}
%\textbf{$S$, $S'$ and $S''$ are basic posets, $\sqsubset\ \subseteq S\times S'$ and $\sqsubset'\ \subseteq S'\times S''$}.
%\end{center}

\begin{dfn}
For any relations $\sqsubset\ \subseteq S\times S'$ and $\sqsubset'\ \subseteq S'\times S''$, define
\[a(\sqsubset\circ\sqsubset')a''\qquad\Leftrightarrow\qquad\exists a'\in S'\ (a\sqsubset a'\sqsubset' a'').\]
\end{dfn}

Note this relation $\sqsubset\circ\sqsubset'\ \subseteq  S\times S''$ is the standard composition of the relations $\sqsubset$ and $\sqsubset'$, which extends the usual notion of function composition.

\begin{prp}
If $\sqsubset\ \subseteq S\times S'$ and $\sqsubset'\ \subseteq S'\times S''$ are basic morphisms between basic posets $S$, $S'$ and $S''$ then their composition $\sqsubset\circ\sqsubset'$ is also a basic morphism.
\end{prp}

\begin{proof}  We verify the required properties of $\sqsubset\circ\sqsubset'$ as follows.
\begin{itemize}
\item[\eqref{Faithful}]  If $a\sqsubset a'\sqsubset'0$ then $a'=0$ and hence $a=0$, by \eqref{Faithful} for $\sqsubset$ and $\sqsubset'$, so \eqref{Faithful} also holds for $\sqsubset\circ\sqsubset'$.

\item[\eqref{Auxiliarity}]  If $a\leq b\sqsubset b'\sqsubset'b''\leq a''$ then $a\sqsubset b'\sqsubset'a''$, by \eqref{Auxiliarity} for $\sqsubset$ and $\sqsubset'$, so \eqref{Auxiliarity} also holds for $\sqsubset\circ\sqsubset'$.

\item[\eqref{Pushforward}]  If $a\prec b\sqsubset b'\sqsubset'c'',d''$ then \eqref{Pushforward} for $\sqsubset$ yields $a'$ with $a\sqsubset a'\prec b'\sqsubset'c'',d''$.  Then \eqref{Pushforward} for $\sqsubset'$ yields $b''$ with $a\sqsubset a'\sqsubset b''\prec c'',d''$, so \eqref{Pushforward} also holds for $\sqsubset\circ\sqsubset'$.

\item[\eqref{veePullback}]  Assume $a\prec b \sqsubset c'\sqsubset'd''\vee e''$.  By \eqref{Interpolation}, we can take $f$ with $a\prec f\prec b\sqsubset c'$.  By \eqref{Pushforward} for $\sqsubset$, we have $b'$ with $f\sqsubset b'\prec c'\sqsubset d''\vee e''$.  By \eqref{veePullback} for $\sqsubset'$, we have $d'\sqsubset'd''$ and $e'\sqsubset'e''$ with $f\sqsubset b'\prec d'\vee e'$.  By \eqref{2Auxiliarity} for $\sqsubset$, $a\prec f\sqsubset d'\vee e'$.  By \eqref{veePullback} for $\sqsubset$, we have $d\sqsubset d'$ and $e\sqsubset e'$ with $a\prec d\vee e$.  In other words, given $a\prec b\sqsubset\circ\sqsubset'd''\vee e''$, we have found $d\sqsubset\circ\sqsubset'd''$ and $e\sqsubset\circ\sqsubset'e''$ with $a\prec d\vee e$, as required.\qedhere
\end{itemize}
\end{proof}

Unfortunately, this does not extend to basic $\vee$-morphisms, as \eqref{veePreserving} and \eqref{LowerRelation} are not always preserved under composition.

\begin{xpl}
To see that this can fail for \eqref{LowerRelation}, let $G=G'=G''$ be the Cantor space $C$, let $S=S''$ be the $\U$-basic poset of all open subsets, let $S'$ be the Boolean algebra of just the clopen subsets, let $\sqsubset$ and $\sqsubset'$ be the restrictions of inclusion $\subseteq$ to $S\times S'$ and $S'\times S''$, and take any non-closed open $O$.  For every open $N\Subset O$, we can find clopen $O'$ with $N\subseteq O'\subseteq O$, even though there is no single clopen $O'$ with $O\subseteq O'\subseteq O$.  Thus $\sqsubset\circ\sqsubset'$ fails to satisfy \eqref{LowerRelation}.
\end{xpl}

\begin{xpl}
For a (necessarily non-Hausdorff) example of a composition failing \eqref{veePreserving}, consider the `bug-eyed interval' $G'=\{0^*\}\cup[0,1]$ where $0^*$ is an extra copy of $0$, i.e. $\{0^*\}\cup(0,1]$ is homoemorphic to $[0,1]$ via the map taking $0^*$ to $0$.  Let $\phi$ be the continuous map from $G=[0,1]$ to $G'$ which maps $[0,\frac{1}{2}]$ to $[0,1]$ and $[\frac{1}{2},1]$ to $\{0^*\}\cup(0,1]$ by going back and forth one time in the obvious way.  In particular,
\[\phi(0)=0\qquad\text{and}\qquad\phi(1)=0^*.\]
Let $\phi'$ be the countinuous map from $G'$ to $G''=[0,1]$ which is the identity on $[0,1]$ and takes $0^*$ to $0$.  Let $S$, $S'$ and $S''$ be the corresponding $\U$-basic posets of all Hausdorff open sets and define basic $\vee$-morphisms $\sqsubset$ and $\sqsubset'$ form $\phi$ and $\phi'$ as in \eqref{sqsubsetdef}.  Then
\[[0,\tfrac{2}{3})\sqsubset[0,1]\sqsubset'[0,1]\qquad\text{and}\qquad(\tfrac{1}{3},1]\sqsubset\{0^*\}\cup(0,1]\sqsubset'[0,1],\]
even though there is no Hausdorff open $O'\subseteq G'$ with $[0,1]\sqsubset O'\sqsubset'[0,1]$.  Thus $\sqsubset\circ\sqsubset'$ fails to satisfy \eqref{veePreserving}.
\end{xpl}

This leaves us with two options:
\begin{enumerate}
\item Just accept that the same continuous map could arise from different basic morphisms.  Given that we already have to accept that the same topology can arise from different bases, this is perhaps not a serious drawback.

\item Take $\underline{\sqsubset\circ\sqsubset'}$ as the product in the category of $\U$-basic posets.  The drawback here is that the definition of $\sqsubseteq$ from $\sqsubset$ is not elementary, as it involves arbitrarily large finite subsets (although it is still first order in the weak sense of being defined by `omitting types' \textendash\, see \cite{Marker2002} Ch 4).
\end{enumerate}

In any case, we should still note that composition of basic morphisms corresponds to composition of the resulting continuous maps.  By \eqref{phidef}, this amounts to
\[U^{\sqsubset\circ\sqsubset'}=(U^\sqsubset)^{\sqsubset'},\]
for all $\prec$-ultrafilters $U\subseteq S$, which is immediate from the definitions.

\section{Inverse Semigroups}\label{InverseSemigroups}

So far we have exhibited a duality between certain posets and topological spaces.  Now we wish to obtain a non-commutative extension to certain inverse semigroups and topological groupoids.\footnote{Strictly speaking, `extension' may not be the right word, as we did not require our posets to have meets, and only the $\wedge$-semilattices can be considered as (commutative) inverse semigroups.}

So let us now assume
\begin{center}
\textbf{$S$ is an inverse semigroup.  We denote the idempotents of $S$ by $E$}.
\end{center}
We always consider $S$ as a poset with respect to the canonical ordering given by
\[a\leq b\qquad\Leftrightarrow\qquad a\in Eb.\]
This allows us to apply all the order theory developed so far.  Indeed, most of the work has already been done, the only thing left to do is investigate how our previous duality interacts with the multiplicative and inverse structure of $S$.

We take \cite{Lawson1998} as our standard reference for inverse semigroups.  A standard example is the symmetric inverse monoid $I(X)$ of all partial bijections on some set $X$ (indeed every inverse semigroup can be represented as a subsemigroup of a symmetric inverse monoid, by the Wagner-Preston theorem \textendash\, see \cite[Theorem 1.5.1]{Lawson1998}).  As far as the order structure is concerned, $I(X)$ is a local Boolean algebra and $\wedge$-semilattice, so $\prec\ =\ \leq$ and $\U\ =\ I(X)\times I(X)$, by \autoref{Uwedge}.

\subsection{Distributivity}\label{ssDistributivity}

\begin{dfn}
We call an inverse semigroup $S$ \emph{distributive} if the product preserves existing joins, i.e. if $a\vee b$ exists then so does $ac\vee bc$ and $ca\vee cb$ and
\[\label{Distributivity}\tag{Distributivity}(a\vee b)c=ac\vee bc\qquad\text{and}\qquad c(a\vee b)=ca\vee cb.\]
\end{dfn}

Actually, one-sided distributivity suffices, as $a\mapsto a^{-1}$ is an automorphism from $S$ to $S^{op}$.  Also, taking $c=(a\vee b)^{-1}$ in \eqref{Distributivity} yields
\begin{equation}\label{aveeb}
(a\vee b)^{-1}(a\vee b)=(a^{-1}a)\vee(b^{-1}b),
\end{equation}
which actually holds even without distributivity \textendash\, see \cite[\S1.4 Proposition 17]{Lawson1998}.  Thus any join of idempotents is idempotent, as noted in \cite[\S1.4 Lemma 14]{Lawson1998}.

However, one thing that appears to have been overlooked in \cite{Lawson1998} is that joins in $E$ are not always joins in $S$.

\begin{xpl}\label{EvsSjoins}
Let $S=\{0,1,e_1,e_2,e_3,s\}$ be the inverse semigroup of partial bijections on $X=\{1,2,3,4,5\}$ where $0$ is the empty function, $1$ is the identity, $e_k$ is the identity restricted to $\{k\}$, and $s$ is the total bijection leaving $1$, $2$ and $3$ fixed but switching $4$ and $5$.  Then $e_1\vee e_2=1$ in $E$ but not in $S$, as $e_1,e_2\leq s\nleq 1$.  Indeed, there are no non-trival joins in $S$ (i.e. $a\vee b$ exists iff $a\leq b$ or $b\leq a$) so $S$ is trivially distributive, even though $E=\{0,1,e_1,e_2,e_3\}$ is the diamond lattice, which is not distributive.
\end{xpl}

Thus the (1)$\Rightarrow$(2) part of \cite[\S1.4 Proposition 20]{Lawson1998} does not actually hold in general.  But it does hold as long as $S$ is a conditional $\vee$-semilattice.

\begin{prp}\label{distributiveidempotents}
If $S$ is a conditional $\vee$-semilattice, the following are equivalent.
\begin{enumerate}
\item\label{Edistributive} $E$ is distributive.
\item\label{Sdistributive} $S$ is distributive.
\item\label{Sleqdistributive} $S$ satisfies \eqref{leqDistributivity}.
\item\label{Sdecomposition} $S$ satisfies \eqref{leqDecomposition}.
\end{enumerate}
\end{prp}

\begin{proof}\
\begin{itemize}
\item[\eqref{Edistributive}$\Rightarrow$\eqref{Sdistributive}]
Take $a,b\in S$ with a join $a\vee b$ (in $S$).  For any $c\in S$, we immediately have $ac,bc\leq(a\vee b)c$.  On the other hand, for any $d\geq ac,bc$,
\begin{align*}
(a\vee b)^{-1}dc^{-1}&\geq(a\vee b)^{-1}acc^{-1}=a^{-1}acc^{-1}.\\
(a\vee b)^{-1}dc^{-1}&\geq(a\vee b)^{-1}bcc^{-1}=b^{-1}bcc^{-1}.
\end{align*}
So by \eqref{aveeb} and \eqref{Distributivity} in $E$,
\[(a\vee b)^{-1}(a\vee b)cc^{-1}=(a^{-1}a\vee b^{-1}b)cc^{-1}=a^{-1}acc^{-1}\vee b^{-1}bcc^{-1}\leq(a\vee b)^{-1}dc^{-1}.\]
Multiplying on the left by $(a\vee b)$ and on the right by $c$ we obtain
\[(a\vee b)c\leq(a\vee b)(a\vee b)^{-1}dc^{-1}c\leq d.\]
As $d$ was arbitrary, $(a\vee b)c$ is a join of $ac$ and $bc$.

\item[\eqref{Sdistributive}$\Rightarrow$\eqref{Sleqdistributive}]  Say $a\leq b\vee c$.  By distributivity, $a=aa^{-1}a=aa^{-1}b\vee aa^{-1}c$.  As $aa^{-1}b=a\wedge b$ and $aa^{-1}c=a\wedge c$, by \cite[\S1.4 Lemma 12]{Lawson1998}, we are done.

\item[\eqref{Sleqdistributive}$\Rightarrow$\eqref{Sdecomposition}]  See \autoref{Distributivity=>Decomposition}.

\item[\eqref{Sdecomposition}$\Rightarrow$\eqref{Edistributive}]  Note it suffices to show that $E$ satisfies \eqref{leqDecomposition'}, as $E$ is a $\wedge$-semilattice in which products are meets.  As $S$ satisfies \eqref{leqDecomposition}, for this it suffices to show that any join in $E$ is a join in $S$.  But if $e\vee f=g$ in $E$ then, in particular, $e$ and $f$ are bounded so we have some $a\in S$ with $e\vee f=a$ in $S$ (note this is where we need the conditional $\vee$-semilattice assumption).  By \eqref{aveeb}, $a\in E$ and hence $g\leq a$.  Conversely, as $E\subseteq S$, $a\leq g$ and hence $a=g$, as required. \qedhere
\end{itemize}
\end{proof}

%\begin{dfn}
%The \emph{compatibility} relation $\sim$ is defined on inverse semigroups by
%\[a\sim b\qquad\Leftrightarrow\qquad ab^{-1},a^{-1}b\in E.\]
%\end{dfn}
%
%Note $\sim$ is somewhat analogous to $\U$, e.g. $\leq\ \subseteq\ \sim$ and
%\[\label{simAuxiliarity}\tag{$\sim$-Auxiliarity}a\leq a'\sim b'\geq b\qquad\Leftrightarrow\qquad a\sim b\]
%(see \cite[\S 1.4 Lemma 13]{Lawson1998}).  With $\sim$ we obtain another generalization of the classical fact that in a distributive lattice, joins also distribute over meets.
%
%\begin{prp}\label{wedgevee}
%If $S$ is a distributive, $a\sim b$ and both $a\vee c$ and $b\vee c$ exist then
%\[(a\vee c)\wedge(b\vee c)=(a\wedge b)\vee c.\]
%\end{prp}
%
%\begin{proof}
%First note that $a\vee c\sim b\vee c$ as
%\[(a\vee c)^{-1}(b\vee c)=((a\vee c)^{-1}b\vee(a\vee c)^{-1}c)=a^{-1}b\vee c^{-1}b\vee c^{-1}c=a^{-1}b\vee c^{-1}c\in E\]
%and, likewise, $(b\vee c)(a\vee c)^{-1}\in E$.  Thus
%\[(a\vee c)\wedge(b\vee c)=(a\vee c)(a\vee c)^{-1}(b\vee c)=(a\vee c)a^{-1}b\vee(a\vee c)c^{-1}c=(a\wedge b)\vee c.\qedhere\]
%\end{proof}

\subsection{Rather Below}\label{ssRatherBelow}

Now assume
\[S\textbf{ is an inverse semigroup with zero}.\]
By definition, this means we have $0\in S$ with $0a=0$, for all $a\in S$.  Thus $0$ is the minimum of $S$ in the canonical ordering.  In particular, we can again define the rather below relation $\prec$ as in \autoref{RatherBelowdfn}.  Although, as before, we will usually assume $S$ is a conditional $\vee$-semilattice and then use the equivalent form in \eqref{RatherBelow}.

In inverse semigroups, $\prec$ is determined, at least partially, by its restriction to $E$.

\begin{prp}\label{Edetermined}
If $S$ is a distributive conditional $\vee$-semilattice and $a\leq b$ then
\[a^{-1}a\prec b^{-1}b\qquad\Rightarrow\qquad a\prec b.\]
\end{prp}

\begin{proof}
Assume $a\leq b$ and $a^{-1}a\prec b^{-1}b$.  If $c\geq b$ then $c^{-1}c\geq b^{-1}b$ so we have $a'\perp a^{-1}a$ with $c^{-1}c\leq a'\vee b^{-1}b$.  Thus $ca'\perp ca^{-1}a\geq ba^{-1}a\geq aa^{-1}a=a$ and $c=cc^{-1}c\leq ca'\vee cb^{-1}b=ca'\vee b$.  As $c$ was arbitrary, $a\prec b$.
\end{proof}

However, the reverse implication is not automatic \textendash\, see \autoref{precvspp}.  For the moment, let us define a stronger relation $\pp$ by
\[a\pp b\qquad\Leftrightarrow\qquad a^{-1}a\prec b^{-1}b\quad\text{and}\quad a\leq b.\]
%Let us also define $\simeq\ =\ \sim\cap\U$ so
%\[a\simeq b\qquad\Leftrightarrow\qquad a\sim b\quad\text{and}\quad a\U b.\]

\subsection{Basic Semigroups}\label{ssBasicSemigroups}

\begin{dfn}
The \emph{compatibility} relation $\sim$ is defined on inverse semigroups by
\[a\sim b\qquad\Leftrightarrow\qquad ab^{-1},a^{-1}b\in E.\]
\end{dfn}

Note $\sim$ is analogous to $\U$, e.g. $\leq\ \subseteq\ \sim$ and, by \cite[\S 1.4 Lemma 13]{Lawson1998},
\[\label{simAuxiliarity}\tag{$\sim$-Auxiliarity}a\leq a'\sim b'\geq b\qquad\Leftrightarrow\qquad a\sim b.\]
Let us denote the combination of $\sim$ and $\U$ by $\simeq\ =\ \sim\cap\U$, i.e.
\[a\simeq b\qquad\Leftrightarrow\qquad a\sim b\quad\text{and}\quad a\U b.\]

\begin{dfn}\label{SimeqBasicSemigroup}
We call $S$ a \emph{$\simeq$-basic semigroup} if $S$ is $\prec$-distributive and
\[\label{CompatibleJoins}\tag{$\simeq$-Joins}a\vee b\text{ exists}\qquad\Leftrightarrow\qquad\exists a',b'\ (a\pp a'\simeq b'\qq b).\]
\end{dfn}

The motivation for \eqref{CompatibleJoins} is that we are imagining $S$ to consist of open Hausdorff bisections that are closed under finite unions `whenever possible', i.e. whenever the union itself is compactly contained in some open Hausdorff bisection \textendash\, see \autoref{EB} \eqref{OcupNEquiv} below.

Again, for the most part, we actually work with slightly more general semigroups.

\begin{dfn}\label{BasicSemigroup}
We call $S$ a \emph{basic semigroup} if
\begin{enumerate}
\item $S$ is $\qq$-round,
\item $S$ is $\prec$-distributive, and
\item $S$ is a conditional $\vee$-semilattice.
\end{enumerate}
\end{dfn}

\begin{prp}
Every $\simeq$-basic semigroup is a basic semigroup.
\end{prp}

\begin{proof}
See the proof of \autoref{Ubasic=>basic}.
\end{proof}

In basic semigroups we can forget about the distinction between $\prec$ and $\pp$.

\begin{prp}\label{precE}
If $S$ is a basic semigroup then $S$ is a basic poset and
\[\label{ESupported}\tag{Bi-Below}a\prec b\qquad\Leftrightarrow\qquad a\pp b.\]
\end{prp}

\begin{proof}
Assume $S$ is a basic semigroup.  In particular, as $S$ is $\qq$-round, $S$ is also $\succ$-round, by \autoref{Edetermined}.  Thus $S$ is basic poset.

If $a\prec b$ then we have $c\geq b$ with $c^{-1}c\succ b^{-1}b$, as $S$ is $\qq$-round.  As $a\prec b\leq c$, we have $d\perp a$ with $c\leq b\vee d$ and hence
\[a^{-1}a\leq b^{-1}b\prec c^{-1}c\leq(b\vee d)^{-1}(b\vee d)=b^{-1}b\vee d^{-1}d.\]
As $a,d\leq b\vee d$ and $d\perp a$, we have $d^{-1}d=(b\vee d)^{-1}d\perp(b\vee d)^{-1}a=a^{-1}a$.  Thus as $a^{-1}a\prec b^{-1}b\vee d^{-1}d$, \eqref{perpDistributivity} yields $a^{-1}a\prec b^{-1}b$ so $a\pp b$.
\end{proof}

It also does not matter if we consider $\prec$ defined within $S$ or within $E$.

\begin{prp}\label{Ebasic}
If $S$ is a basic semigroup then $E$ is a basic poset and
\[e\prec f\text{ in }E\qquad\Leftrightarrow\qquad e\prec f\text{ in }S.\]
\end{prp}

\begin{proof}
Take $e,f\in E$ with $e\prec f$ (in $S$).  For any $g\in E$ with $g\geq f$, we have $h\in E$ with $h\succ g$, as $S$ is $\qq$-round.  Then \eqref{Complements} yields $e'\perp e$ with $g\prec e'\vee f\prec h$.  Thus $f\in E$, as $f\leq h\in E$.  As $g$ was arbitrary, this shows $e\prec f$ also holds in $E$.

Thus $E$ is $\succ$-round, as $S$ is $\qq$-round.  Also, $E$ is a conditional $\vee$-semilattice as $S$ is.  In fact, by \eqref{aveeb}, joins in $S$ of elements of $E$ must also lie in $E$.  This means \eqref{precDistributivity} in $S$ implies \eqref{precDistributivity} in $E$.  Thus $E$ is a basic poset.

Conversely, say $e\prec f$ in $E$.  Then we have $g\in E$ with $f\prec g$ in $S$, as $S$ is $\qq$-round.  In particular, $f\leq g$ so we have some $e'\in E$ with $e'\perp e$ and $g\leq e'\vee f$ (noting joins in $E$ are joins in $S$ because $S$ is a conditional $\vee$-semilattice).  Thus in $S$ we have $e\leq f\prec g\leq e'\vee f$ so $e\prec f$, by \eqref{perpDistributivity}.
\end{proof}

\begin{prp}\label{acprecbd}
If $S$ is a basic semigroup then
\[a\prec b,\quad c\prec d\quad\text{and}\quad b^{-1}b\U dd^{-1}\qquad\Rightarrow\qquad ac\prec bd.\]
\end{prp}

\begin{proof}
If $a,b,c,d\in S$ satisfy the given conditions, $a^{-1}a\prec b^{-1}b$ and $cc^{-1}\prec dd^{-1}$, by \eqref{ESupported}.  Applying \eqref{wedgePreservation} to $E$ yields
\[a^{-1}acc^{-1}=a^{-1}a\wedge cc^{-1}\prec b^{-1}b\wedge dd^{-1}=b^{-1}bdd^{-1}.\]
By \autoref{Edetermined}, $acc^{-1}\prec bdd^{-1}$ or, equivalently, $cc^{-1}a^{-1}\prec dd^{-1}b^{-1}$.  Thus $acc^{-1}a^{-1}\prec bdd^{-1}b^{-1}$, by \eqref{ESupported}, so $ac\prec bd$, by \autoref{Edetermined}.
\end{proof}

\begin{cor}
If $S$ is a basic semigroup then
\[\label{Invariance}\tag{Invariance}a\prec b\quad\text{and}\quad b^{-1}b\leq cc^{-1}\qquad\Rightarrow\qquad ac\prec bc.\]
\end{cor}

\begin{proof}
Assume $a\prec b$ and $b^{-1}b\leq cc^{-1}$.  By \eqref{succRound}, we have $d\succ c$ and hence $dd^{-1}\geq cc^{-1}\geq b^{-1}b$.  As $\leq\ \subseteq\ \U$, we have $b^{-1}b\smallsmile dd^{-1}$ so $ac\prec bd=bc$, by \autoref{acprecbd}.
\end{proof}

\section{Groupoids}\label{Groupoids}

\subsection{Topology}

\[\textbf{From now on we assume $G$ is both a groupoid and a topological space}.\]

In other words, $G$ is a topological space on which we have a partially defined multiplication $\cdot$ which turns $G$ into a small category in which every point is an isomorphism.  Generally, we also want the topology to behave nicely with respect products and inverses, for which we recall the following standard definitions.

\begin{dfn}
We call the groupoid $G$
\begin{enumerate}
\item \emph{topological} if $g\mapsto g^{-1}$ and $(g,h)\mapsto gh$ are continuous.
\item \emph{\'{e}tale} if, moreover, $g\mapsto g^{-1}g$ is a local homeomorphism.
\end{enumerate}
\end{dfn}

First, we want to extend \autoref{LHCneighbourhoods} to `bi-Hausdorff bisections'.

\begin{dfn}
We denote the unit space of $G$ by $G^0$.  We call $B\subseteq G$
\begin{enumerate}
\item a \emph{bisection} if $B^{-1}B\subseteq G^0$ and $BB^{-1}\subseteq G^0$.
\item \emph{bi-Hausdorff} if $B^{-1}B$ and $BB^{-1}$ are Hausdorff.
\end{enumerate}
\end{dfn}

Equivalently, $B$ is a bisection iff $g\mapsto g^{-1}g$ and $g\mapsto gg^{-1}$ are injective on $B$.  Indeed, if $gh^{-1}\in G^0$ and $g^{-1}g=h^{-1}h$ then $g^{-1}g=g^{-1}gh^{-1}h=g^{-1}h$ so $g=h$.  Note that we are not requiring bisections to be open.

\begin{prp}\label{LHCbisections}
If $G$ is \'{e}tale and locally Hausdorff then any compact bi-Hausdorff bisection $C\subseteq G$ is contained in an open Hausdorff bisection.
\end{prp}

\begin{proof}
For every $g\in C$, take some open Hausdorff bisection $O_g\ni g$.  For every $h\in C\setminus O_g$, we have open Hausdorff bisections $N_h\ni g$ and $M_h\ni h$.  As $C$ is a bi-Hausdorff bisection, $g^{-1}g\neq h^{-1}h$ and $gg^{-1}\neq hh^{-1}$ can be separated by open sets.  As $g\mapsto g^{-1}g$ and $g\mapsto gg^{-1}$ are continuous, this means we can make $N_h$ and $M_h$ smaller if necessary to ensure that
\[(N_h^{-1}\cdot N_h)\cap(M_h^{-1}\cdot M_h)=\emptyset=(N_h\cdot N_h^{-1})\cap(M_h\cdot M_h^{-1})\]
so $N_h\cup M_h$ is also an open Hausdorff bisection.  As $C$ is compact and $O_g$ is open, $C\setminus O_g$ is also compact so we have some finite subcover $M'_g=M_{h_1}\cup\ldots\cup M_{h_k}$ of $C\setminus O_g$.  Let $N'_g=O_g\cap N_{h_1}\cap\ldots\cap N_{h_k}$.  Again as $C$ is compact, we have some finite subcover $N=N'_{g_1}\cup\ldots\cup N'_{g_j}$ of $C$.  Let
\[O=N\cap(O_{g_1}\cup M_{g_1})\cap\ldots\cap(O_{g_j}\cap M_{g_j}).\]
Certainly $O$ is an open subset containing $C$.  To see that $O$ is a Hausdorff bisection, take any $g,h\in O$.  As $g\in N$, we have $g\in N'_{g_l}$ for some $l$.  We also have $h\in O_{g_l}\cup M'_{g_l}$.  If $h\in O_{g_l}$ then, as $g\in N'_{g_l}\subseteq O_{g_l}$ too and $O_{g_l}$ is Hausdorff, we can separate $h$ and $g$ with disjoint open sets.  Moreover, as $O_{g_l}$ is a bisection, $h^{-1}h\neq g^{-1}g$ and $hh^{-1}\neq gg^{-1}$.  If $h\in M'_{g_l}$ then $h\in M_{h_j}$ for some $j$ so, as $N'_{g_l}\cup M_{h_j}\subseteq N_{h_j}\cup M_{h_j}$ is a Hausdorff bisection, the same argument applies.  As $g$ and $h$ were arbitrary, we are done.
\end{proof}

In topological groupoids, multiplication preserves compactness as long as $G^0$ is Hausdorff, which will be important later on in \autoref{G0H=>Uetalebasis}.

\begin{prp}\label{CompactProduct}
If $G$ is topological and $G^0$ is Hausdorff, the product $C\cdot D\subseteq G$ of any compact $C,D\subseteq G$ is again compact.
\end{prp}

\begin{proof}
It suffices to show that any net $(g_\lambda)\subseteq C\cdot D$ has a subnet $(g_\alpha)$ converging to some $g\in C\cdot D$.  But if $(g_\lambda)\subseteq C\cdot D$ then we have nets $(c_\lambda)\subseteq C$ and $(d_\lambda)\subseteq D$ with $g_\lambda=c_\lambda d_\lambda$.  As $C$ and $D$ are compact, we have subnets $(c_\alpha)$ and $(d_\alpha)$ with $c_\alpha\rightarrow c\in C$ and $d_\alpha\rightarrow d\in D$ (find a convergent subnet $(c_\alpha)$ first, then take a further subnet so that $(d_\alpha)$ converges too).  Thus $c_\alpha^{-1}c_\alpha\rightarrow c^{-1}c$ and $d_\alpha d_\alpha^{-1}\rightarrow dd^{-1}$.  Moreover, $c_\alpha^{-1}c_\alpha=d_\alpha d_\alpha^{-1}\in G^0$, for all $\alpha$, so we must have $c^{-1}c=dd^{-1}$, as $G^0$ is Hausdorff and limits are unique in Hausdorff spaces.  Thus $cd$ is defined and $g_\alpha=c_\alpha d_\alpha\rightarrow cd\in C\cdot D$, as required.
\end{proof}

\subsection{\'{E}tale Bases}

\begin{dfn}\label{EB}
We call a basis $S$ of $G$ an \emph{\'{e}tale basis} if, for all $O,N\in S$,
\begin{enumerate}
\item\label{O-1N} $O^{-1}\in S$.
\item\label{OcdotN} $O\cdot N\in S$.
\item\label{O-1cdotO} $O^{-1}\cdot O\subseteq G^0$.
\end{enumerate}
We call $S$ a \emph{$\cup$-\'{e}tale basis} if \eqref{O-1cdotO} is replaced by
\[\tag{3$'$}\label{OcupN}O\cup N\in S\ \Leftrightarrow\ O\cup N\subseteq B,\text{ for some compact bi-Hausdorff bisection }B\subseteq G.\]
\end{dfn}

Note that \eqref{O-1N} and \eqref{O-1cdotO} imply that every element of an \'{e}tale basis must be a bisection.  Likewise, as with \eqref{UnionClosed} for $\cup$-bases, taking $O=N$ in the $\Rightarrow$ part of \eqref{OcupN}, we see that every element of a $\cup$-\'{e}tale basis is contained in a compact bi-Hausdorff bisection.  Then the $\Leftarrow$ part of \eqref{OcupN} ensures that $S$ is a conditional $\vee$-semilattice.  By \autoref{LHCbisections}, we could also restate \eqref{OcupN} using $\Subset$ as
\[\label{OcupNEquiv}\tag{3$'$}O\cup N\in S\quad\Leftrightarrow\quad O\cup N\Subset B,\text{ for some open Hausdorff bisection }B\subseteq G.\]

\begin{prp}\label{EG<=>EB}
$G$ has an \'{e}tale basis if and only if $G$ is an \'{e}tale groupoid.
\end{prp}

\begin{proof}
If $G$ is an \'{e}tale groupoid then we claim that the collection $S$ of all open bisections is an \'{e}tale basis.  To see this, note that if $O$ is a bisection then so is $O^{-1}$.  Also, as $g\mapsto g^{-1}$ is continuous it is also homeomorphism, begin an involution.  Thus $O^{-1}\in S$ whenever $O\in S$.  Likewise, if $O,N\in S$ then $O\cdot N$ is a bisection which is also open, by \cite[Lemma 2.4.11]{Sims2017}, i.e. $O\cdot N\in S$.

Conversely, assume $S$ is an \'{e}tale basis of $G$.  As $O^{-1}\in S$, for all $O\in S$, the map $g\mapsto g^{-1}$ is continuous.  Whenever $gh\in M\in S$, we can take $O\in S$ with $g\in O$.  Then $h\in O^{-1}\cdot M\in S$ so $gh\in O\cdot (O^{-1}\cdot M)\subseteq M$. Thus multiplication is also continuous and hence $G$ is topological.  As $O\cdot N\in S$, for all $O,N\in S$, multiplication is also an open map.  Moreover, for any $e\in G^0$, we have some $O\in S$ with $e\in O$.  Then $O^{-1}\cdot O\in S$ and $e\in O^{-1}\cdot O\subseteq G^0$, so $G^0$ is open.  Thus $G$ is \'{e}tale, by \cite[Theorem 5.18]{Resende2007}.
\end{proof}

Any \'etale basis for $G$ forms an inverse semigroup under pointwise operations.  The resulting inverse semigroup relations have natural interpretations within $G$.

\begin{prp}\label{EtaleBasisleqsim}
If $S$ is an \'etale basis then, for any $O,N\in S$,
\begin{align}
\label{OleqN}O\leq N\qquad&\Leftrightarrow\qquad O\subseteq N.\\
\label{OsimN}O\sim N\qquad&\Leftrightarrow\qquad O\cup N\text{ is a bisection}.
\end{align}
\end{prp}

\begin{proof}\
\begin{itemize}
\item[\eqref{OleqN}]  If $O\subseteq N$ then, for each $g\in O$, we have $g=gg^{-1}g\in O\cdot O^{-1}\cdot N$ so $O\subseteq O\cdot O^{-1}\cdot N$.  Conversely, if $fg^{-1}h$ is defined, for some $f,g\in O$ and $h\in N$ then, as $g\in O\subseteq N$ and $N$ is a bisection, we must have $h=g$.  Likewise, $f=g$ so $fg^{-1}h=gg^{-1}g=g$ and hence $O\cdot O^{-1}\cdot N\subseteq O$.  Thus $O=O\cdot O^{-1}\cdot N$, which means $O\leq N$.

\item[\eqref{OsimN}]  If $O\cdot N^{-1}$ and $O^{-1}\cdot N$ are idempotents in $S$ then they must consist only of units of $G$, i.e. $O\cdot N^{-1},O^{-1}\cdot N\subseteq G^0$.  As $O$ and $N$ are bisections, we also have $O\cdot O^{-1},O^{-1}\cdot O,N\cdot N^{-1},N^{-1}\cdot N\subseteq G^0$ and hence
\[(O\cup N)\cdot(O\cup N)^{-1},(O\cup N)^{-1}\cdot(O\cup N)\subseteq G^0,\]
i.e. $O\cup N$ is a bisection.\qedhere
\end{itemize}
\end{proof}

\begin{prp}\label{Uetalebasis=>LCLHEG}
If $G$ has a $\cup$-\'{e}tale basis $S$ then $G$ is a locally compact locally Hausdorff \'{e}tale groupoid.
\end{prp}

\begin{proof}
By \autoref{EG<=>EB}, $G$ is an \'{e}tale groupoid.  By the $O=N$ case of \autoref{EB} \eqref{OcupN}, we see that every $O\in S$ is contained in a compact bi-Hausdorff bisection $B$.  As $G$ is topological, $B$ and hence $O$ must be Hausdorff.  Indeed, for any distinct $g,h\in B$, $g^{-1}g\neq h^{-1}h$ can be separated by disjoint open sets.  This means $g$ and $h$ can be separated by disjoint open sets, as $g\mapsto g^{-1}g$ is continuous.  As each $g\in G$ is contained in some $O\in S$, $G$ is locally compact locally Hausdorff, by \autoref{LCLHchars}.
\end{proof}

Unfortunately, the converse does not hold in general.

\begin{xpl}\label{noUetalebasis}
We construct a locally compact locally Hausdorff \'{e}tale groupoid $G$ such that every \'{e}tale basis of $G$ contains a set with no compact extension.

To construct such a $G$, consider the bug-eyed Cantor space $C$ given in \autoref{bug-eyedC} (the bug-eyed interval would also do), denoting the two `eyes' by $x$ and $y$.  Let $B=C\sqcup C'\sqcup C''$ consist of three copies of $C$ and consider the equivalence relation $G\subseteq B\times B$ defined as follows.  First,
\[G\cap(C\times C)\cap(C'\times C')\cap(C''\times C'')\ =\ \Delta_B\cap\ (C\times C)\cap(C'\times C')\cap(C''\times C''),\]
i.e. the restriction of $G$ to $C$, $C'$ or $C''$ is just the equality relation.  Also, for $c\in C$,
\begin{align}
(c,c')\in G\qquad&\Leftrightarrow\qquad c\in C\setminus\{y\}.\\
(c,c'')\in G\qquad&\Leftrightarrow\qquad c\in C\setminus\{x\}.\\
\label{b'b''}(c',c'')\in G\qquad&\Leftrightarrow\qquad c\in C\setminus\{x,y\}.
\end{align}
This completely determines the equivalence relation $G$, which becomes a groupoid in a standard way.  Moreover, $G$ is topological as a subspace of $B\times B$.  One can also check that $G$ has the local homeomorphisms necessary to be \'{e}tale.

Now any \'{e}tale basis $S$ of $G$ must include open sets $O$ and $N$ with
\begin{align*}
(x,x') &\in O\subseteq C\times C'.\\
(y,y'') &\in N\subseteq C\times C''.
\end{align*}
As $S$ is \'{e}tale, $O^{-1}\cdot N\in S$.  As $x$ and $y$ are not isolated in $C$, we have a sequence $(c_n)\subseteq C\setminus\{x,y\}$ with $c_n\rightarrow x$ or, equivalently, $c_n\rightarrow y$.  Thus $(c_n,c_n')\rightarrow(x,x')$ so $(c_n,c_n')$ is eventually in $O$.  Likewise $( c_n,c_n'')$ is eventually in $N$, as $(c_n,c_n'')\rightarrow(y,y'')$.  Thus $(c_n',c_n'')=(c_n',c_n)\cdot(c_n,c_n'')$ is eventually in $O^{-1}\cdot N$.  However, $(c_n',c_n'')$ has no limit in $G$, as \eqref{b'b''} excludes all of the potential limits $( x',x'')$, $(y',x'')$, $(x',y'')$ and $(y',y'')$.   This means no subnet of $(c_n',c_n'')$ has a limit in $G$ either so $O^{-1}\cdot N$ is not contained in any compact subset of $G$.

In particular, the $\Rightarrow$ part of \autoref{EB} \eqref{OcupN} fails, so $G$ has no $\cup$-\'{e}tale basis.
\end{xpl}

Note too that the above $G$ is \'{e}tale and has a basis of compact open subsets, even though it does not have an \'{e}tale basis of compact open subsets.  Although a simpler example of such a $G$ would be the bug-eyed Cantor space itself, seen as an \'{e}tale groupoid with trivial multiplication.

The non-Hausdorff nature of \autoref{noUetalebasis} was crucial.  Indeed, as long as the unit space $G^0$ is Hausdorff, we do have a converse of \autoref{Uetalebasis=>LCLHEG}.

\begin{rmk}\label{GeneralizationRemark}
If one really does want to work with general locally compact locally Hausdorff groupoids $G$ where $G^0$ may not be Hausdorff, this would likely still be feasible.  On the one hand, one would replace \autoref{EB} \eqref{OcupN} with
\begin{align*}
O\cup N\in S\qquad\Leftrightarrow\qquad&(O\cup N)^{-1}\cdot(O\cup N)\subseteq B\text{ and }(O\cup N)\cdot(O\cup N)^{-1}\subseteq C,\\
&\text{for some compact Hausdorff }B,C\subseteq G^0.
\end{align*}
On the other hand, one would replace \eqref{CompatibleJoins} with
\begin{align*}
a\vee b\text{ exists}\qquad\Leftrightarrow\qquad&a\sim b\text{ and }\exists e,f,g,h\\
&(a^{-1}a\prec e\U f\succ b^{-1}b\quad\text{and}\quad aa^{-1}\prec g\U h\succ bb^{-1}).
\end{align*}
Also $\Subset$ would now correspond to $\pp$ rather than $\prec$.  While somewhat messier, essentially the same duality could probably be obtained.  However, $G^0$ is Hausdorff for most \'{e}tale groupoids of interest to operator algebraists and we already have sufficient generality for these, by \autoref{G0H=>Uetalebasis}.
\end{rmk}

\begin{prp}\label{G0H=>Uetalebasis}
If $G$ is \'{e}tale and $G^0$ is locally compact and Hausdorff then $G$ has a $\cup$-\'{e}tale basis $S$.
\end{prp}

\begin{proof}
Simply let $S$ be the collection of open subsets of $G$ contained in some (automatically bi-Hausdorff) compact bisection.  As $G$ is \'{e}tale and $G^0$ is locally compact, $S$ is a basis of $G$.  As $g\mapsto g^{-1}$ is a homeomorphism, $S$ also satisfies \autoref{EB} \eqref{O-1N}.  As $G$ is \'{e}tale and $G^0$ is Hausdorff, $S$ also satisfies \autoref{EB} \eqref{OcdotN}, by \autoref{CompactProduct}.  The definition of $S$ also immediately yields \autoref{EB} \eqref{OcupN}.
\end{proof}

As in \autoref{Ubases}, we obtain obtain general examples of $\simeq$-basic semigroups from $\cup$-\'{e}tale bases (the proof is essentially the same).

\begin{thm}\label{EtaleBases}
Let $S$ be a $\cup$-\'{e}tale basis for some (necessarily) \'etale groupoid $G$. Then $S$ is a $\simeq$-basic semigroup.  Also, for any $O,N\in S$,
\begin{align}
\label{OprecN2}O\prec N\qquad&\Leftrightarrow\qquad O\Subset N.\\
\label{OUN2}O\U N\qquad&\Leftrightarrow\qquad O\cup N\text{ is Hausdorff}.
\end{align}
\end{thm}

It might seem reasonable to guess that $\pp$ and $\prec$ should coincide in any inverse semigroup which is a basic poset.  We now give an elementary counterexample, showing that $S$ really needs to be $\qq$-round, not just $\succ$-round, for this to hold.

\begin{xpl}\label{precvspp}
First let $G$ be the space consisting of two copies of the unit interval, i.e. $G=[0,1]\sqcup[0,1]'$.  We turn this into an \'{e}tale groupoid by defining $r\cdot r=r'\cdot r'=r$ and $r\cdot r'=r'\cdot r=r'$, for each $r\in[0,1]$ and corresponding $r'\in[0,1]'$.  Let $S$ be the collection of all open bisections $O$ of $G$ such that $0'$ is not in the boundary of $O$, i.e. such that
\[0'\in\overline{O}\qquad\Rightarrow\qquad 0'\in O.\]
Then $S$ is an $\cup$-\'{e}tale basis of $G$ and hence a $\simeq$-basic semigroup.  In particular, $S$ is a basic poset, by \autoref{precE} (and $E$ is also a basic poset, by \autoref{Ebasic}).

Now consider the groupoid $H=G\setminus\{0'\}$.  The map $\phi(O)=O\cap H$ takes $S$ to an inverse semigroup $T$ of open bisections of $H$.  Moreover, $\phi$ is an order isomorphism and even a $\sim$-isomorphism, although not a semigroup homomorphism, as \[\phi([0,1]'\cdot[0,1]')=\phi([0,1])=[0,1]\neq(0,1]=(0,1]'\cdot(0,1]'=\phi([0,1]')\cdot\phi([0,1]').\]
In particular, both $T$ and its idempotents form basic posets too and $T$ even still satisfies \eqref{CompatibleJoins}.  However, $T$ is not $\qq$-round.  Indeed, $a=(0,1]'$ is a maximal element of $T$, so trivially $a\prec a$, even though $e\not\prec e$ for $e=aa^{-1}=a^{-1}a=(0,1]$, as $e\subseteq[0,1]\ni0$ and there is no open subset disjoint from $(0,1]$ which contains $0$.
\end{xpl}

\subsection{Lenz Groupoids}

In the next section we will examine a natural groupoid structure on $\prec$-ultrafilters.  But first, we consider general $\leq$-filters.
\begin{center}
\textbf{Throughout this subsection, let $S$ be an arbitrary inverse semigroup.}
\end{center}
By \cite[Theorem 3.1]{Lenz2008}, the set of all $\leq$-filters on $S$ becomes an inverse semigroup under the multiplication operation $\cdot$ given by
\[T\cdot U=(TU)^\leq.\]
Restricting multiplication to the case when
\[(T^{-1}T)^\leq=(UU^{-1})^\leq,\]
we obtain the \emph{Lenz groupoid} of $S$ \textendash\, see \cite[\S3.1]{Lawson1998} and \cite[\S3.3]{LawsonMargolisSteinberg2013}.

We wish to give some characterizations of when the product is defined in the Lenz groupoid \textendash\, see \autoref{TUdefined}.  First, as in \cite[Lemma 3.3]{LawsonMargolisSteinberg2013}, we note that $U=UU^{-1}U$, for any filter $U$.  Indeed, $U\subseteq UU^{-1}U$ holds for arbitrary $U\subseteq S$, while the reverse inclusion follows from the following slightly more general result.

\begin{prp}
If $T,U\subseteq S$ and $U$ is a filter then
\begin{equation}\label{TUUsubT}
T^{-1}T\subseteq(UU^{-1})^\leq\qquad\Leftrightarrow\qquad T^{-1}TU\subseteq U.
\end{equation}
\end{prp}

\begin{proof}
If $T^{-1}T\subseteq(UU^{-1})^\leq$ then, for any $a,b\in T$ and $c\in U$, we have $a',b'\in U$ with $a'b'^{-1}\leq a^{-1}b$.  As $U$ is a filter, we have $c'\in U$ with $c'\leq a',b',c$ so $c'=c'c'^{-1}c'\leq a'b'^{-1}c\leq a^{-1}bc$ and hence $a^{-1}bc\in U$.  Thus $T^{-1}TU\subseteq U$.  Conversely, if $T^{-1}TU\subseteq U$ then, for any $a,b\in T$ and $c\in U$, we have $a^{-1}bc\in U$ so $a^{-1}b\geq a^{-1}bcc^{-1}\in UU^{-1}$, i.e. $a^{-1}b\in(UU^{-1})^\leq$.  Thus $T^{-1}T\subseteq(UU^{-1})^\leq$.
\end{proof}

Also note that, as long as $T\neq\emptyset$ and $U$ is upwards closed,
\begin{equation}\label{TTUleq}
(T^{-1}TU)^\leq=U\qquad\Leftrightarrow\qquad T^{-1}TU\subseteq U.
\end{equation}
Indeed, taking $a\in T$, we have $U\subseteq(a^{-1}aU)^\leq\subseteq(T^{-1}TU)^\leq$.  Also $(T^{-1}TU)^\leq\subseteq U$ implies $T^{-1}TU\subseteq U$ while, conversely, $T^{-1}TU\subseteq U$ implies $(T^{-1}TU)^\leq\subseteq U^\leq=U$.

We also note that $T^{-1}TU\subseteq U$ implies that the product $(TU)^\leq$ can be calculated from any single element of $T$.

\begin{prp}\label{aU=TU}
If $a\in T$ and $T^{-1}TU\subseteq U$ then $(aU)^\leq=(TU)^\leq$.
\end{prp}

\begin{proof}
It suffices to show that $(aU)^\leq\subseteq(bU)^\leq$, for all $a,b\in T$.  But as $T^{-1}TU\subseteq U$, this follows from $(aU)^\leq\subseteq(bb^{-1}aU)^\leq\subseteq(bT^{-1}TU)^\leq\subseteq(bU)^\leq$.
\end{proof}

The following shows that a product in a filter can be used to express the filter itself as a product in the Lenz groupoid.

\begin{prp}\label{abW}
If $W$ is a filter with $ab\in W$ then we have filters $U=(Wb^{-1})^\leq$ and $V=(a^{-1}W)$ with $(U^{-1}U)^\leq=(VV^{-1})^\leq$ and $W=(UV)^\leq$.
\end{prp}

\begin{proof}
As $ab\in W$, we have $abb^{-1}a^{-1}\in WW^{-1}$ so $(abb^{-1}a^{-1}W)^\leq=W$, by \eqref{TUUsubT} and \eqref{TTUleq} (taking $T=\{b^{-1}a^{-1}\}$).  Thus
\[(bb^{-1}a^{-1}W)^\leq=(bb^{-1}a^{-1}aa^{-1}W)^\leq=(a^{-1}abb^{-1}a^{-1}W)^\leq=(a^{-1}W)^\leq.\]
Also $(W^{-1}W)^\leq=(b^{-1}a^{-1}W)^\leq$, by \eqref{TUUsubT} and \autoref{aU=TU}, from which we obtain $(bW^{-1}W)^\leq=(bb^{-1}a^{-1}W)^\leq=(a^{-1}W)^\leq$.  This and $W=WW^{-1}W$ yields
\[(bW^{-1}Wb^{-1})^\leq=(bW^{-1}WW^{-1}Wb^{-1})^\leq=(a^{-1}WW^{-1}a)^\leq,\]
i.e. $(U^{-1}U)^\leq=(VV^{-1})^\leq$.  Thus $(UV)^\leq=(abb^{-1}V)^\leq=(abb^{-1}a^{-1}W)^\leq=W$, again by \eqref{TUUsubT} and \autoref{aU=TU}.
\end{proof}

\begin{prp}\label{TUdefined}
The following are equivalent, for any filters $T,U\subseteq S$.
\begin{enumerate}
\item\label{TTneqUU} $(T^{-1}T)^\leq=(UU^{-1})^\leq$.
\item\label{TTUU} $T^{-1}T\cap UU^{-1}=T^{-1}TUU^{-1}$.
\item\label{TFU} $TU\cap TZU=\emptyset$, where $Z=S\setminus(T^{-1}T\cap UU^{-1})^\leq$.
\end{enumerate}
\end{prp}

\begin{proof}\
\begin{itemize}
\item[\eqref{TTneqUU}$\Rightarrow$\eqref{TTUU}]  If $a\in T^{-1}T\cap UU^{-1}$ then $a=aa^{-1}a\in T^{-1}TT^{-1}TUU^{-1}=T^{-1}TUU^{-1}$, i.e. $T^{-1}T\cap UU^{-1}\subseteq T^{-1}TUU^{-1}$ whenever $T$ (or $U$) is a filter.  Now say $(T^{-1}T)^\leq=(UU^{-1})^\leq$.  In particular, $UU^{-1}\subseteq (T^{-1}T)^\leq$ so $TUU^{-1}\subseteq T$, by \eqref{TUUsubT}, and hence $T^{-1}TUU^{-1}\subseteq T^{-1}T$.  Likewise, $T^{-1}TUU^{-1}\subseteq UU^{-1}$ and hence $T^{-1}TUU^{-1}\subseteq T^{-1}T\cap UU^{-1}$.

\item[\eqref{TTUU}$\Rightarrow$\eqref{TTneqUU}]  If $T^{-1}TUU^{-1}\subseteq(T^{-1}T)^\leq$ then, for any $u\in U$,
\[(T^{-1}T)^\leq\subseteq(T^{-1}Tuu^{-1})^\leq\subseteq(T^{-1}TUU^{-1})^\leq\subseteq(T^{-1}T)^\leq.\]
Likewise if $T^{-1}TUU^{-1}\subseteq(UU^{-1})^\leq$ then, for any $t\in T$,
\[(UU^{-1})^\leq\subseteq(t^{-1}tUU^{-1})^\leq\subseteq(T^{-1}TUU^{-1})^\leq\subseteq(UU^{-1})^\leq.\]
Thus if $T^{-1}TUU^{-1}\subseteq(T^{-1}T\cap UU^{-1})^\leq$ then
\[(T^{-1}T)^\leq=(T^{-1}TUU^{-1})^\leq=(UU^{-1})^\leq.\]

\item[\eqref{TTUU}$\Rightarrow$\eqref{TFU}]  Say \eqref{TFU} fails, so we have $ab=a'cb'$ for some $a,a'\in T$, $b,b'\in U$ and $c\notin(T^{-1}T\cap UU^{-1})^\leq$.  In fact we can assume $a=a'$ and $b=b'$, as $T$ and $U$ are filters (taking $a''\in T$ with $a''\leq a,a'$ and $b''\in U$ with $b''\leq b,b'$, we obtain $a''b''=a''a''^{-1}abb''^{-1}b''=a''a''^{-1}a'cb'b''^{-1}b''=a''cb''$).   Then $c\geq a^{-1}acbb^{-1}=a^{-1}abb^{-1}$ so $c\in(T^{-1}TUU^{-1})^\leq$.  In particular, it follows that $T^{-1}T\cap UU^{-1}\neq T^{-1}TUU^{-1}$, i.e. \eqref{TTUU} also fails.

\item[\eqref{TFU}$\Rightarrow$\eqref{TTUU}]  Say \eqref{TTUU} and hence \eqref{TTneqUU} fails so $(T^{-1}T)^\leq\neq (UU^{-1})^\leq$.  W.l.o.g. we can assume $(T^{-1}T)^\leq\nsubseteq(UU^{-1})^\leq$.  As $(T^{-1}T)^\leq=\{t^{-1}t:t\in T\}^\leq$, we can take $t\in T$ with $t^{-1}t\notin(UU^{-1})^\leq$.  For any $u\in U$, we have $tu=tt^{-1}tu$ where $t^{-1}t\in E\setminus(UU^{-1})^\leq\subseteq S\setminus(T^{-1}T\cap UU^{-1})^\leq$, i.e. \eqref{TFU} also fails. \qedhere
\end{itemize}
\end{proof}

%As $\leq$ preserves multiplication in $S$, i.e. $a\leq b$ and $a'\leq b'$ implies $aa'\leq bb'$, it follows that $UV$ is downwards directed whenever $U,V\subseteq S$ are.  Thus $(UV)^\leq$ is a filter whenever $U$ and $V$ are.  In fact, this operation turns the set of all filters into another inverse semigroup.  

%for any $U,V\subseteq S$ we have
%\[(UV)^\leq=(U^\leq V)^\leq=(UV^\leq)^\leq=(U^\leq V^\leq)^\leq.\]

\subsection{Ultrafilter Groupoids}\label{ssStoneSpaces}

In \autoref{SubsetUltrafilters} we showed how each point $g$ in a space corresponds to an ultrafilter $U_g=\{O\in S:g\in O\}$ of elements in a basis $S$.  In \'etale groupoids, multiplying points also corresponds to multiplying the corresponding ultrafilters.

\begin{prp}\label{UltrafilterMultiplication}
If $S$ is a basis of an \'etale groupoid $G$ then $U_{gh}=(U_gU_h)^\subseteq$,
\[\text{i.e.}\quad\{O\in S:gh\in O\}=\{O\in S:g\in M\in S,h\in N\in S\text{ and }MN\subseteq O\},\]
whenever $gh$ is defined.  Moreover, consider the following statements.
\begin{enumerate}
\item\label{gh} $gh$ is defined.
\item\label{UgUh} $\bigcap U_gU_h\neq\emptyset$.
\item\label{UgUgUhUh} $(U_g^{-1}U_g)^\subseteq=(U_hU_h^{-1})^\subseteq$.
\end{enumerate}
Generally, \eqref{gh} $\Rightarrow$ \eqref{UgUh} and \eqref{UgUgUhUh}.  If $G$ is $T_0$, \eqref{gh} $\Leftrightarrow$ \eqref{UgUgUhUh}.  If $G$ is $T_1$, \eqref{gh} $\Leftrightarrow$ \eqref{UgUh}.
\end{prp}

\begin{proof}
If $g\in M\in S$, $h\in N\in S$ and $MN\subseteq O$ then certainly $gh\in O$, i.e. $(U_gU_h)^\subseteq\subseteq U_{gh}$.  Conversely, if $gh\in O\in S$ then, as $G$ is \'etale and $S$ is a basis, we have some bisection $N\in U_h$, i.e. $h\in N\in S$ so $g=ghh^{-1}\in O\cdot N^{-1}$.  As $G$ is \'etale, the product is an open map so $O\cdot N^{-1}$ is open.  As $S$ is a basis, we thus have some $M\in U_g$, i.e. $g\in M\in S$, such that $M\subseteq O\cdot N^{-1}$.  Then
\[M\cdot N\subseteq O\cdot N^{-1}\cdot N\subseteq O.\]
As $O$ was arbitrary, this shows that $U_{gh}\subseteq(U_gU_h)^\subseteq$, as required.

\begin{itemize}
\item[\eqref{gh}$\Rightarrow$\eqref{UgUh}]  If $gh$ is defined then $gh\in\bigcap U_gU_h$.

\item[\eqref{gh}$\Rightarrow$\eqref{UgUgUhUh}]  If $gh$ is defined, $g^{-1}g=hh^{-1}$ so $(U_g^{-1}U_g)^\subseteq=U_{g^{-1}g}=U_{hh^{-1}}=(U_hU_h^{-1})^\subseteq$.

\item[\eqref{UgUgUhUh}$\Rightarrow$\eqref{gh}]  If $gh$ is not defined then $g^{-1}g\neq hh^{-1}$, so if $G$ is $T_0$ then we have some $O\in S$ distinguishing $g^{-1}g$ and $hh^{-1}$, i.e. $O$ is in precisely one of $U_{g^{-1}g}$ and $U_{hh^{-1}}$ and hence $(U_g^{-1}U_g)^\subseteq=U_{g^{-1}g}\neq U_{hh^{-1}}=(U_hU_h^{-1})^\subseteq$.

%W.l.o.g. assume $g^{-1}g\in O\not\ni hh^{-1}$.  As $G$ is \'etale, $G^0$ is open so we may assume $O\subseteq G^0$ (by replacing $O$ with $O\cap G^0$ if necessary).  Then $g\in G\cdot O\not\ni hh^{-1}$ and hence we have some $M\in S$ with $g\in M\subseteq G\cdot O$.  Then $g^{-1}g\in M^{-1}\cdot M\subseteq G\cdot O\not\ni hh^{-1}$ so $M^{-1}\cdot M\in U_g^{-1}U_g$ even though $M^{-1}\cdot M\notin U_{hh^{-1}}=(U_hU_h^{-1})^\subseteq$.  Thus $(U_g^{-1}U_g)^\subseteq\neq(U_hU_h^{-1})^\subseteq$.

\item[\eqref{UgUh}$\Rightarrow$\eqref{gh}]  Assume $G$ is $T_1$.  If $gh$ is not defined then again $g^{-1}g\neq hh^{-1}$ so we have some $O\in S$ with $g^{-1}g\in O\not\ni hh^{-1}$.  As $G$ is \'etale, we have a bisection $M\in U_g$ with $M^{-1}\cdot M\subseteq O\not\ni hh^{-1}$.  For any $f\in M$, $f^{-1}f\neq hh^{-1}$ so, again as $G$ is $T_1$, we have some $O'\in S$ with $hh^{-1}\in O'\not\ni f^{-1}f$.  Then we have a bisection $N\in U_h$ with $N\cdot N^{-1}\subseteq O'\not\ni f^{-1}f$.  Thus $\{f\}\cdot N=\emptyset$, as the product with $f$ is not defined for any element of $N$.  As $M$ is a bisection, this means $M\cdot N$ does not contain any element of the form $fk$, for $k\in N$.  As $f$ was arbitrary, this shows that $\bigcap U_gU_h\subseteq\bigcap\{M\cdot N:N\in U_h\}=\emptyset$. \qedhere
\end{itemize}
\end{proof}

%$G^0$ is open so by taking smaller $O\in S$ if necessary we may assume $O\subseteq G^0$.  Then $g\in G\cdot O\not\ni hh^{-1}$ and hence we have some $M\in S$ with $g\in M\subseteq G\cdot O$.  Then $g^{-1}g\in M^{-1}\cdot M\subseteq G\cdot O\not\ni hh^{-1}$ so $M^{-1}\cdot M\in U_g^{-1}U_g$ even though $M^{-1}\cdot M\notin U_{hh^{-1}}=(U_hU_h^{-1})^\subseteq$.

%, as above, we have some bisection $M\in U_g$, i.e. $g\in M\in S$, with $M^{-1}\cdot M\not\ni hh^{-1}$.

Our primary goal in this section is to show that we can also turn the $\prec$-ultrafilter space of a basic semigroup into an \'etale groupoid by multiplying $\prec$-ultrafilters.

\begin{center}
\textbf{Let $S$ be a basic semigroup and let $G$ be the set of $\prec$-ultrafilters of $S$.}
\end{center}

%Recall that $U^\leq=\{v\geq u:u\in U\}$, for any $U\subseteq S$.  As $\leq$ preserves multiplication, i.e. $a\leq b$ and $a'\leq b'$ implies $aa'\leq bb'$, for any $U,V\subseteq S$ we have
%\[(UV)^\leq=(U^\leq V)^\leq=(UV^\leq)^\leq=(U^\leq V^\leq)^\leq.\]

\begin{prp}\label{SGaction}
For any $a\in S$ and $U\in G$, the following are equivalent.
\begin{enumerate}
\item\label{aUsubVinG} $aU\subseteq V\in G$, for some $V$.
\item\label{aUinG} $(aU)^\leq\in G$.
\item\label{aa-1UU-1} $a^{-1}a\in(UU^{-1})^\leq$.
\item\label{0notinaU} $a^{-1}a\in(UU^{-1})^\U$ and $0\notin bU$, for some $b\prec a$.
\end{enumerate}
\end{prp}

\begin{proof}\
\begin{itemize}
\item[\eqref{0notinaU}$\Rightarrow$\eqref{aUsubVinG}]  Say $a^{-1}a\in(UU^{-1})^\U$, $0\notin bU$ and $b\prec a$.  So $a^{-1}a\U uv^{-1}$, for some $u,v\in U$.  As $U$ is $\leq$-directed, we have some $w\in U$ with $w\leq u,v$ so $ww^{-1}\leq uv^{-1}$ and hence $a^{-1}a\U ww^{-1}$, by \eqref{UAuxiliarity}.

We claim $V=\bigcup_{b\prec c\leq a}cU\supseteq aU$ is $\prec$-directed.  As $U$ is a $\prec$-filter, we can replace $U$ with $U^w$, i.e. it suffices to show $\{cu:b\prec c\leq a,w\geq v\in U\}$ is $\prec$-directed.  But $\{c:b\prec c\leq a\}$ is $\prec$-directed, by \eqref{LocallyHausdorff} and \eqref{UInterpolation}, and $U^w$ is $\prec$-directed, as $U$ is a $\prec$-filter.  Thus \autoref{acprecbd}, $a^{-1}a\U ww^{-1}$ and \eqref{UAuxiliarity} imply that their product is also $\prec$-directed.  Thus $V$ can be extended to a $\prec$-ultrafilter.

\item[\eqref{aUsubVinG}$\Rightarrow$\eqref{aUinG}]  Assume $aU\subseteq V\in G$ and take any $u\in U$.  Then
\[U\subseteq(a^{-1}aU)^\leq\subseteq(a^{-1}V)^\leq=(a^{-1}V^{au})^\leq.\]
By \eqref{Invariance}, $a^{-1}V^{au}$ is $\prec$-directed so $(a^{-1}V^{au})^\leq$ is a proper $\prec$-filter.  By the maximality of $U$, we have equality above, i.e. $U=(a^{-1}V)^\leq$ so
\[V\subseteq(aa^{-1}V)^\leq=(a(a^{-1}V)^\leq)^\leq=(aU)^\leq.\]
By the maximality of $V$, equality holds again, i.e. $(aU)^\leq=V\in G$.

\item[\eqref{aUinG}$\Rightarrow$\eqref{aa-1UU-1}]  Assume $V=(aU)^\leq\in G$.  Then as above, for any $u\in U$,
\[U\subseteq(a^{-1}aU)^\leq\subseteq(a^{-1}V)^\leq=(a^{-1}V^{au})^\leq.\]
The maximality of $U$ again yields equality so $U=(a^{-1}aU)^\leq$.  In particular, $a^{-1}au\in U$ so $a^{-1}a\geq a^{-1}auu^{-1}a^{-1}a\in UU^{-1}$, i.e. $a^{-1}a\in(UU^{-1})^\leq$.

\item[\eqref{aa-1UU-1}$\Rightarrow$\eqref{0notinaU}]  Assume $a^{-1}a\in(UU^{-1})^\leq\subseteq(UU^{-1})^\U$, i.e. $a^{-1}a\geq uv^{-1}$ for some $u,v\in U$.  Taking $w\in U$ with $w\leq u,v$, we have $ww^{-1}\leq a^{-1}a$.  Taking $x\in U$ with $x\prec w$, \eqref{Invariance} yields both $xx^{-1}=xw^{-1}\prec ww^{-1}\leq a^{-1}a$ and $b=axx^{-1}\prec aa^{-1}a=a$.  For any $u\in U$, we have $v\in U$ with $v\leq u,x$ and hence $bu=axx^{-1}u\geq avv^{-1}v=av$.  As
\[v^{-1}a^{-1}av\geq v^{-1}xx^{-1}v\geq v^{-1}vv^{-1}v=v^{-1}v\neq0,\]
we have $av\neq 0$ and hence $bu\neq 0$.\qedhere
\end{itemize}
\end{proof}

In a similar vein, we have the following.

\begin{cor}
For any $U,V\in G$, the following are equivalent.
\begin{enumerate}
\item\label{UVsubW} $UV\subseteq W\in G$, for some $W$.
%, in which case
%\begin{equation}\label{uVUv}
%W=(UV)^\leq=(uV)^\leq=(Uv)^\leq.
%\end{equation}
\item\label{U-1U=V-1V} $(U^{-1}U)^\leq=(VV^{-1})^\leq$
%, in which case
%\begin{equation}\label{U-1UcapV-1V}
%(U^{-1}U\cap VV^{-1})^\leq=(U^{-1}U)^\leq=(VV^{-1})^\leq.
%\end{equation}
\item\label{0UV} $0\notin UV$ and $u^{-1}u\smallsmile vv^{-1}$, for some $u\in U$ and $v\in V$.
\end{enumerate}
\end{cor}

\begin{proof}\
\begin{itemize}
\item[\eqref{0UV}$\Rightarrow$\eqref{UVsubW}]  As $U$ and $V$ are $\leq$-filters, $UV$ is $\leq$-directed.  If $u^{-1}u\smallsmile vv^{-1}$, for some $u\in U$ and $v\in V$, then $UV$ is even $\prec$-directed.  Indeed, for any $u'\in U$ and $v'\in V$ we have $u'',u'''\in U$ and $v'',v'''\in V$ with $u'''\prec u''\leq u',u$ and $v'''\prec v''\leq v',v$ so, by \autoref{acprecbd}, $u'''v'''\prec u''v''\leq u'v'$.  Thus we can extend $UV$ to some $\prec$-ultrafilter $W$.

\item[\eqref{UVsubW}$\Rightarrow$\eqref{U-1U=V-1V}]  If $UV\subseteq W\in G$ then, for any $a,b\in U$, we have $c\in U$ with $c\leq a,b$.  Then $cV\subseteq W$ so \autoref{SGaction} yields $a^{-1}b\geq c^{-1}c\in(VV^{-1})^\leq$.  As $a$ and $b$ were arbitrary, this shows that $U^{-1}U\subseteq(VV^{-1})^\leq$.  Likewise $VV^{-1}\subseteq(U^{-1}U)^\leq$ and hence $(U^{-1}U)^\leq=(VV^{-1})^\leq$.

\item[\eqref{U-1U=V-1V}$\Rightarrow$\eqref{0UV}]  If $(U^{-1}U)^\leq=(VV^{-1})^\leq$ then, for all $u\in U$, we have $u^{-1}u\in(VV^{-1})^\leq$ so \autoref{SGaction} yields $u^{-1}u\in(VV^{-1})^\U$ and $0\notin uV$.  As $u$ was arbitrary, $0\notin UV$. \qedhere
\end{itemize}
\end{proof}

\begin{prp}
$G$ is a subgroupoid of the Lenz groupoid of $S$.
\end{prp}

\begin{proof}
We need to show that
\[U,V\in G\quad\text{and}\quad(U^{-1}U)^\leq=(VV^{-1})^\leq\qquad\Rightarrow\qquad(UV)^\leq\in G.\]
For any $u\in U$, $(U^{-1}U)^\leq=(VV^{-1})^\leq$ implies $u^{-1}u\in(VV^{-1})^\leq$.  As $V\in G$, \autoref{SGaction} then yields $(uV)^\leq\in G$.  But $(UV)^\leq=(uV)^\leq$, by \eqref{TUUsubT} and \autoref{aU=TU}.
\end{proof}

%Assume $(U\cdot V)\cdot W$ is defined.  In particular, $0\notin(U\cdot V)^\leq W\supseteq UVW$ so $0\notin VW$ and $0\notin(UVW)^\leq\supseteq U(VW)^\leq$.  Also $u^{-1}u\smallsmile vv^{-1}$ and $v'^{-1}u'^{-1}u'v'\smallsmile ww^{-1}$, for some $u,u'\in U$, $v,v'\in V$ and $w\in W$.  As $U\cdot V$ is defined, $x=u'^{-1}u'v'\in V$ by the proof of \eqref{UVsubW}$\Rightarrow$\eqref{U-1U=V-1V} above.  As $x^{-1}x=v'^{-1}u'^{-1}u'v'\smallsmile ww^{-1}$, $V\cdot W$ is defined.  As $u^{-1}u\smallsmile vv^{-1}\geq vww^{-1}v$, $U\cdot(V\cdot W)$ is also defined.
%
%For any $u,v\in U$, we have $w\in U$ with $w\leq u,v$ and hence $0\neq w^{-1}w\leq u^{-1}v$, i.e. $0\notin U^{-1}U$.  Also \eqref{LocallyHausdorff} implies $(u^{-1})^{-1}u^{-1}=uu^{-1}\smallsmile uu^{-1}$ so $U^{-1}\cdot U$ is defined.  Moreover, if $U\cdot V$ is defined then, again by the proof of \eqref{UVsubW}$\Rightarrow$\eqref{U-1U=V-1V} above,
%\[U^{-1}\cdot U\cdot V=(U^{-1}UV)^\leq=((U^{-1}U)^\leq V)^\leq=(u^{-1}uV)^\leq=V.\]
%Symmetric arguments show that $V\cdot V^{-1}$ is defined and $U\cdot V\cdot V^{-1}=U$.

\begin{thm}\label{SemigroupRep}
Let $S$ be a basic semigroup and let $G$ be the set of $\prec$-ultrafilters of $S$. Then $G$ is a locally compact locally Hausdorff \'{e}tale groupoid and
\begin{align*}
O_a^{-1}&=O_{a^{-1}}.\\
O_{ab}&=O_a\cdot O_b.
\end{align*}
Moreover, if $S$ is a $\simeq$-basic semigroup then $(O_a)_{a\in S}$ is an $\cup$-\'{e}tale basis for $G$.
\end{thm}

\begin{proof}
By \autoref{SpacesFromPosets}, $G$ is locally compact locally Hausdorff.

If $a\in U$ and $b\in V$ then certainly $ab\in UV$ so $O_a\cdot O_b\subseteq O_{ab}$.  Conversely, if $ab\in W\in G$ then $b\in(a^{-1}W)^\leq\in G$, by \autoref{SGaction}, as $aa^{-1}\geq abb^{-1}a^{-1}\in WW^{-1}$.   Likewise, $a\in(Wb^{-1})^\leq\in G$.  Setting $U=(Wb^{-1})^\leq$ and $V=(a^{-1}W)^\leq$, we have $(U^{-1}U)^\leq=(VV^{-1})^\leq$ and $W=(UV)^\leq$, by \autoref{abW}, and hence $W\in O_a\cdot O_b$.  As $W$ was arbitrary, $O_{ab}\subseteq O_a\cdot O_b$.  Thus $(O_a)_{a\in S}$ is an \'{e}tale basis and hence $G$ is an \'{e}tale groupoid, by \autoref{EG<=>EB}.

If $S$ is a $\simeq$-basic semigroup then we can show that $(O_a)_{a\in S}$ is a $\cup$-\'{e}tale basis by essentially the same argument as at the end of the proof of \autoref{SpacesFromPosets}.
\end{proof}

%\eqref{TUUsubT} and \autoref{aU=TU} yield
%\[U\cdot V=(UV)^\leq=(aV)^\leq=(aa^{-1}W)^\leq=W.\]

This completes the duality as far as the objects are concerned.  As in \autoref{secFunctoriality}, this duality is also functorial with respect to the appropriate morphisms, namely partial continuous functors between \'{e}tale groupoids (i.e. contiuous maps preserving multiplication whenever it is defined) and basic ($\vee$-)morphisms $\sqsubset$ between basic semigroups which also preserve multiplication and inverses, i.e. also satsifying
\begin{align*}
a\sqsubset a'\quad\text{and}\quad b\sqsubset b'\qquad&\Rightarrow\qquad ab\sqsubset a'b'.\\
a\sqsubset a'\qquad&\Rightarrow\qquad a^{-1}\sqsubset a'^{-1}.
\end{align*}

\bibliography{maths}{}
\bibliographystyle{alphaurl}

\end{document}